\documentclass[12pt]{amsart}
\usepackage{amsfonts}
\usepackage{amsmath}
\usepackage{amssymb}
\usepackage[margin=1in]{geometry}
\usepackage{tikz}
\usepackage{amsthm}   							
\usepackage{mathtools}							
\usepackage{commath}							
\usepackage{enumitem}  							

\usepackage[english]{babel}						

\usepackage{hyperref}

\makeatletter
\@namedef{subjclassname@2020}{\textup{2020} Mathematics Subject Classification}
\makeatother

\newtheorem{theorem}{Theorem}[section]

\newtheorem{lemma}[theorem]{Lemma}
\newtheorem{proposition}[theorem]{Proposition}

\theoremstyle{definition}
\newtheorem{definition}[theorem]{Definition}
\newtheorem{remark}[theorem]{Remark}

\numberwithin{equation}{section}

\newcommand{\N}{\mathbb{N}}

\newcommand{\R}{\mathbb{R}}
\newcommand{\C}{\mathbb{C}}

\renewcommand{\Re}{\operatorname{Re}}
\renewcommand{\Im}{\operatorname{Im}}

\newcommand{\I}{\mathrm{i}}
\newcommand{\e}{\mathrm{e}}

\newcommand{\eps}{\varepsilon}
\newcommand{\vphi}{\varphi}

\newcommand{\D}{\mathcal{D}}
\newcommand{\MS}{\mathcal{S}}
\newcommand{\ML}{\mathcal{L}}
\newcommand{\F}{\mathcal{F}}

\newcommand{\fracD}[3]{\prescript{}{#1}{D}^{#2}_{#3}}	

\newcommand{\mus}{\mu_{\sigma}}
\newcommand{\mue}{\mu_{\eps}}
\newcommand{\nues}{\nu_{\eps,\sigma}}

\DeclareMathOperator{\supp}{supp}

\DeclareMathOperator{\sgn}{sgn}

\begin{document}

\title{Distributed-order time-fractional wave equations}

\author[F. Broucke]{Frederik Broucke}
\thanks{Frederik Broucke was supported by the Ghent University BOF-grant 01J04017}

\author[{Lj}. Oparnica]{Ljubica Oparnica}
\thanks{Ljubica Oparnica was supported by the FWO Odysseus 1 grant no. G.0H94.18N: Analysis and Partial Differential Equations}

\address{Department of Mathematics: Analysis, Logic and Discrete Mathematics\\ Ghent University\\ Krijgslaan 281\\ 9000 Gent\\ Belgium}
\email{fabrouck.broucke@ugent.be}
\email{oparnica.ljubica@ugent.be}

\begin{abstract}
Distributed-order time-fractional wave equations appear in the modeling of wave propagation in viscoelastic media. The material characteristics of the medium are modeled through constitutive functions or distributions in the distributed-order constitutive law.  
In this work we propose to take positive Radon measures for the constitutive ``functions''.
 First, we derive a thermodynamical restriction on the constitutive measures which is easy to check, and therefore suitable for applications. Then we prove that the setting with measures in combination with the derived thermodynamical restriction guarantee existence and uniqueness of solutions for the distributed-order fractional wave equation. We further discuss the support and regularity of the fundamental solution, and conclude with a discussion on wave velocities.
\end{abstract}



\maketitle

\section{Introduction}
The distributed-order time-fractional wave equation describes waves occurring in viscoelastic media and is obtained from the system of basic equations of elasticity: the equation of motion coming from Newton's second law, the stress-strain relation given via the distributed-order constitutive equation, and the strain measure:
\begin{align}
	\dpd{}{x}\sigma 	(x,t)					&= \dpd[2]{}{t} u(x,t),  \label{eq: NL} \\
	\int_{0}^{1}\phi_{\sigma}(\alpha)\fracD{0}{\alpha}{t}\sigma(x,t)\dif \alpha
									&= \int_{0}^{1}\phi_{\eps}(\alpha)\fracD{0}{\alpha}{t}\eps(x,t)\dif\alpha, \label{eq: const. equation}\\
	\varepsilon (x,t)						&=  \dpd{}{x} u(x,t). \label{eq: sm}
\end{align}
Here $\sigma$, $u$, and $\varepsilon$ denote stress, displacement, and strain, respectively, considered as functions of $x\in \mathbb{R}$ and $t>0$; $\fracD{0}{\alpha}{t}$, $0 \le \alpha \le 1$, is the left Riemann-Liouville operator of fractional differentiation; 
 and $\phi_{\sigma}$ and $\phi_{\eps}$ are so-called \emph{constitutive functions}. 
 %
 %

The distributed-order constitutive law for a viscoelastic body (\ref{eq: const. equation}) was proposed in \cite{Atanackovic2002} as a generalization of previously known and widely used linear constitutive laws for viscoelastic media, such as the (fractional) Zener constitutive law, the classical and the fractional Maxwell or Voigt models, or the linear constitutive laws investigated in \cite{KOZ2011} (see also \cite{Mainardi} for a detailed and modern introduction to the modeling of viscoelastic materials via fractional differential operators). 

The left (resp.\ right) hand side of \eqref{eq: const. equation} is called a distributed-order fractional derivative (DOFD) of $\sigma$ (resp.\ $\eps$). A DOFD represents a weighted average of derivatives of different fractional orders. Due to their far-reaching applicability, the distributed-order fractional calculus is a rapidly emerging branch within the field of fractional calculus. For a detailed review of the work in the field of DOFDs and applications we refer to \cite{DOreview}. 
The main scientific branches that find successful application of modeling with DOFDs are the fields in which already (ordinary) fractional calculus have proved to be successful. In viscoelasticity, which is our case, DOFDs are used for material characterization, in control theory for enhancing the flexibility and robustness of controllers, and furthermore transport processes,  anomalous, reaction-diffusion and advection processes are successfully modelled with DOFDs. Although our work is related to applications in viscoelasticity, the results could be adapted for applications in other fields.
%
%

Under suitable conditions as shown in \cite{KOZ2019}, 
the system (\ref{eq: NL})-(\ref{eq: sm}) gives rise to the following equation for the displacement $u$, called the distributed-order time-fractional wave equation: 
\begin{equation}
\label{eq: DFWE}
\dpd[2]{}{t}u(x,t) = L(t)\ast_{t}\dpd[2]{}{x}u(x,t), \qquad x\in\R, \quad t>0.
\end{equation} 
Here, $L$ is a convolution kernel depending on the constitutive functions $\phi_{\sigma}$ and $\phi_{\eps}$ and is of the form given below by \eqref{eq: L}. 

In \cite{KOZ2019}, the constitutive law \eqref{eq: const. equation} was employed in the distributional setting, where one allows the constitutive ``functions'' $\phi_{\sigma}$, $\phi_{\eps}$ to be compactly supported distributions (the integral then being replaced by the distributional action \eqref{DefInt}, see \cite{AOP2009} for more details). In \cite{KOZ2019}, six technical conditions  (reviewed later in Subsection \ref{Sec: Six conditions})  are proposed, which allow one to derive the distributed-order fractional wave equation \eqref{eq: DFWE}, show existence and uniqueness of solutions to the corresponding Cauchy problem, establish support properties of the fundamental solution, and obtain an integral representation of this solution.

The distributional approach generalizes many concrete fractional constitutive models. However, it has the drawback that for every case, six technical conditions (stated in terms of certain integral transforms of $\phi_{\sigma}$ and $\phi_{\eps}$) have to be verified. A second important issue concerns finding restrictions which ensure that the constitutive law is thermodynamically acceptable. Such a condition was examined in \cite{Atanackovic2003}, but again, the condition is a technical one (stated in terms of the transforms of $\phi_{\sigma}$ and $\phi_{\eps}$).

\bigskip

The aim of this paper is threefold. First, we propose a constitutive equation of the form \eqref{eq: const. equation} where $\phi_{\sigma}$ and $\phi_{\eps}$ are Radon measures instead of general distributions. We propose a single general thermodynamical condition for these constitutive measures, which unifies all previously known thermodynamical restrictions, and which is easy to verify in most applications. We proceed to show that under this thermodynamical condition, the six technical conditions from \cite{KOZ2019} hold.

Next, we discuss smoothness of the fundamental solution. We for example show that for a large subclass of models, one has smoothness (even Gevrey regularity) on the boundary of the forward light-cone.

Finally, we discuss wave velocities. We introduce the concepts of weak initial and weak equilibrium velocity, which are useful notions of ``wave packet speed'' at small and large times, respectively. We compute them in terms of quantities associated to the constitutive measures, and relate them to the material constants of the viscoelastic body.

\bigskip

The paper is organized as follows. In the remainder of this section we provide some mathematical preliminaries and fix notations. In Section \ref{Sec: DoDmeasures} we propose the distributed-order fractional derivative using a positive Radon measure as weight ``function'', and analyze its main properties. Next, we investigate in Section \ref{Sec: TD} distributed-order constitutive equations using the in this way defined distributed-order derivative, i.e.\ with measures as constitutive ``functions". In Section \ref{Sec: existence and uniqueness} we prove that all conditions necessary for the unique existence of solutions to the Cauchy problem for the distributed-order fractional wave equation are satisfied, as well as those conditions providing representation formulas and support properties of the fundamental solution. In Section \ref{Sec: Smooth} we investigate the smoothness of the fundamental solution, and finally in Section \ref{Sec: QA} we discuss some qualitative aspects regarding wave velocities.

\subsection{Some notations and definitions}
For $\Omega\subset\R^{n}$ we denote by $\mathcal{D}(\Omega)$ the space of compactly supported smooth functions on $\Omega$,  by $\mathcal{D}'(\Omega)$ the space of distributions on $\Omega$ and by $\mathcal{E}'(\Omega)$ the space of compactly supported distributions on $\Omega$. The Schwartz space of rapidly decreasing smooth functions is denoted by $\mathcal{S}(\R^{n})$ and its dual space,  the space of tempered distributions, by $\mathcal{S}'(\R^{n})$. Finally, $\mathcal{D}'_{+}(\R) \subset \mathcal{D}'(\R)$ and $\mathcal{S}'_{+}(\R) \subset \mathcal{S}'(\R)$  denote subspaces of distributions supported on $[0,\infty)$, and $\mathcal{S}'(\mathbb{R}\times \mathbb{R}_{+})$ denotes the space of distributions in $\mathcal{S}'(\mathbb{R}^{2})$ vanishing on $\R\times(-\infty, 0)$.

\subsubsection*{Integral transforms} The Fourier transform for an integrable function $\varphi \in L^1(\R)$
 is
\[
	 \F\varphi (\xi)=\hat{\varphi}(\xi)=\int_{-\infty }^{\infty}\varphi (x)\e^{-\I\xi x}\dif x, \qquad \xi \in \mathbb{R},
\]
while for  $u\in \mathcal{S}'(\mathbb{R})$ the Fourier transform is given via $\left\langle \F u,\varphi \right\rangle =\left\langle u,\F\varphi \right\rangle$, $\varphi \in \mathcal{S}(\mathbb{R})$.

 The Laplace transform of $u\in \mathcal{D}'_{+}(\mathbb{R})$ satisfying $\e^{-at}u\in \mathcal{S}'(\mathbb{R})$, for all $a >a_{0}>0$ is given by
\[
 	\ML u(s) =\tilde{u}(s) = \F(\e^{-a t}u)(y), \qquad s=a+\I y, \qquad a>a_{0},
\]
and defines a holomorphic function in the half plane $\Re s>a_{0}$. 
If $u\in\MS'(\R_{+})$, then $\ML u(s) = \langle u(t), \e^{-st}\rangle$. In particular, for 
$u\in L^1(\mathbb{R})$ with $u(t)=0$, for $t<0$, 
the Laplace transform is given by
\[
 \ML u(s)=\int_{0}^{\infty} u(t)\e^{-st}\dif t, \qquad \Re s\ge0.
\]
For a function $F$ holomorphic in the half plane $\Re s> 0$ and satisfying for some $m,k\in \N$ the bound
\begin{equation}\label{LTbounds}
\abs{F(s)} \leq A\frac{(1+\,\abs{s})^{m}}{\abs{\Re s}^{k}}, \quad \Re s> 0,
\end{equation} 
the inverse Laplace transform exists as distribution in $\MS'_{+}(\R)$  and  it is given by
\begin{equation} \label{eq: Laplace inverse formula}
 \ML^{-1}F(t) =\lim_{Y\to\infty}\frac{1}{2\pi\I}\int_{a-\I Y}^{a+\I Y} F(s) \e^{st}\dif s, \qquad t>0, \quad a>0,
\end{equation}
whenever this limit exists.

If $f(x,t) \in \MS'(\R\times\R_{+})$, then the Laplace transform of $f$ with respect to $t$ is the distribution-valued function
\[
	\ML_{t}f: \{s: \Re s>0\} \to \MS'(\R): s \mapsto \bigl(\phi(x) \mapsto \langle f(x,t), \phi(x)\e^{-st}\rangle\bigr).
\]

\subsubsection*{Fractional derivatives}  There is extensive literature in the field of fractional calculus. For a comprehensive overview on the theory we refer to the classical source \cite{SKM}, while we refer to \cite{APSZ-2}  for the specific applications in mechanics, and to \cite{Mainardi} for applications in viscoelasticity. Here, for the convenience of the reader, we give the definitions used in our model and some important formulas. 

The left Riemann-Liouville fractional derivative of order $\alpha\in [0,1)$ is defined for an absolutely continuous function, $f\in AC([0,a])$, on an interval $[0,a]$ with $a>0$ by
\[
 	\fracD{0}{\alpha}{t}f(t) =\frac{1}{\Gamma (1-\alpha)} \od{}{t} \int_{0}^{t} \frac{f(\zeta)}{(t-\zeta)^{\alpha}}\dif\zeta, \quad t\in [0,a],
\]
 where $\Gamma$ denotes the Euler gamma function. 
 
For $u\in  \mathcal{S}_+' $ the (left) Riemann-Liouville fractional derivative of order $\alpha$ is defined via convolution with a member from the family $\{f_{\alpha}\}_{\alpha\in\R}$\footnote{$f_\alpha(t)=H(t)t^{\alpha-1}/\Gamma(\alpha)$, $\alpha>0$, and $f_\alpha(t)= (\frac{d}{dt})^N f_{\alpha+N}(t)$, for $\alpha<0$ with $\alpha+N>0$.} for $\alpha<0$, as described for example in \cite[Section 2]{AOP2009}, where a framework for the analysis of equations with fractional derivatives in the spaces of (tempered) distribution was developed. For $u\in AC([0,a])\subset \MS'_+(\R)$ we have $\fracD{0}{\alpha}{t}u(t) = f_{-\alpha}\ast u (t)$.
 The Laplace transform of  a Riemann-Liouville fractional derivative of $u\in \MS'_+$ is
 \[
 \ML[\fracD{0}{\alpha}{t}u](s) = s^{\alpha}\ML u(s).
\]

In \cite[Proposition 2.1]{AOP2009} it is proved that for fixed $\alpha \in \R $ the mapping $u\mapsto \fracD{0}{\alpha}{t}u: \mathcal{S}_+'\to  \mathcal{S}_+'$ is linear and continuous, for fixed $u\in  \mathcal{S}_+'$  the mapping $\alpha\mapsto \fracD{0}{\alpha}{t}u: \R\to  \mathcal{S}_+'$ is smooth, and the mapping $(\alpha,u)\mapsto \fracD{0}{\alpha}{t}u: \R\times \mathcal{S}_+'\to  \mathcal{S}_+'$ is continuous.
This makes the following definition of distributed-order derivative well-defined. 
\subsection{Distributed-order derivative and distributed-order differential equation}
Let $\phi\in \mathcal{E}'(\R)$ and $u\in\mathcal{S}^{\prime}_+(\mathbb{R})$. The distributed-order fractional derivative of $u$ with weight $\phi$, $D^\phi u$,  is the element of $\mathcal{S}^{\prime}_+(\mathbb{R})$ defined via
\begin{equation}  
\label{DefInt}
\langle D^\phi u, \psi \rangle 	\coloneqq \Big\langle \int_{\supp \phi}\phi(\alpha) \fracD{0}{\alpha}{t} u(t) \dif\alpha, \psi(t) \Big\rangle
						\coloneqq \big\langle \phi(\alpha) , \langle \fracD{0}{\alpha}{t}u(t) , \psi(t) \rangle \big\rangle, \quad \psi\in\mathcal{S} (\mathbb{R}).
\end{equation}
When $\supp\phi\subseteq [c,d]$ one writes $D^{\phi}u(t) = \int_c^d\phi(\alpha){}_0D_t^{\alpha}u(t) \dif\alpha$. 
Special cases are continuous functions $\phi$ of $\alpha$ in $[c,d]$, for which the distributional evaluation $\langle \phi(\alpha), \dotso\rangle$ is ordinary integration,
and linear combinations of Dirac delta distributions $\phi(\alpha)=\sum_{i=0}^{k} c_{i}\delta(\alpha -\alpha_{i})$, $\alpha_{i}\in \mathbb{R}$, $i\in\{0,1,...,k\}$, for which the distributed-order fractional derivative reduces to a finite linear combination of Riemann-Liouville fractional derivatives, i.e., $\sum_{i=0}^{k}c_{i}\,\, {}_0D_t^{\alpha_i}u(t)$, cf. \cite{AOP2009}.

The mapping $u\mapsto D^\phi u$ is linear and continuous from $\mathcal{S}_+'$ to $\mathcal{S}_+'$ and the Laplace transform of distributed-order fractional derivative $D^\phi u$ is calculated as 
\begin{equation}\label{LTdofd}
{\mathcal{L} } [D^\phi u](s) = {\mathcal{L} }\Big(\int_{\supp \phi} \phi(\alpha)\fracD{0}{\alpha}{t}u(t)\dif\alpha \Big)(s) = \tilde{u}(s)\langle \phi(\alpha) , s^{\alpha} \rangle, \qquad \mathop{\rm Re}\nolimits s>0.
\end{equation}

If $\phi$ is an absolutely integrable function with compact support, and $\int_{\supp \phi} \phi(\alpha) s^\alpha \dif\alpha$ is nonzero for $\Re s>0$, then for a given integrable function $f$ it could be shown, for example by the means of the Laplace transform, that the distributed-order fractional differential equation
\begin{equation}\label{DOE}
D^\phi u = f
\end{equation}
has a unique (bounded) absolutely continues solution $u$, cf.\ \cite{DOreview,FM2012,DF2009}. 
\subsection{Six conditions} \label{Sec: Six conditions}
Let the constitutive functions $\phi_{\sigma}$, $\phi_{\eps}$ in the general constitutive equation \eqref{eq: const. equation} be compactly supported distributions with support in $[0,1]$, and set $\Phi_{\sigma}(s) = \langle\phi_{\sigma}(\alpha),s^{\alpha}\rangle$ and $\Phi_{\eps}(s) = \langle\phi_{\eps}(\alpha),s^{\alpha}\rangle$. The existence and uniqueness of solutions for the wave equation \eqref{eq: DFWE}, as well as a representation formula for the fundamental solution and support properties, were established in \cite[Section 3]{KOZ2019} under the following assumptions on the functions $\Phi_{\sigma}$, $\Phi_{\eps}$:
\begin{enumerate}[label=\textbf{(A\arabic*)}]
\item $\ML^{-1}(\Phi_{\eps}(s)/\Phi_{\sigma}(s))$ exists as an element of $\mathcal{S}_{+}'$; \label{A1}
\item $s^{2}\Phi_{\sigma}(s)/\Phi_{\eps}(s) \in \C \setminus (-\infty, 0]$ for all $s\in\C$ with $\Re s>0$; \label{A2}
\item $\ML^{-1}(\Phi_{\sigma}(s)/\Phi_{\eps}(s))$ exists as an element of $\mathcal{S}_{+}'$; \label{A3}
\item $\sqrt{\Phi_{\sigma}(s)/\Phi_{\eps}(s)}$ has at most the two branch points $s=0$ and $s=\infty$; \label{A4}
\item $\lim_{\,\abs{s}\to\infty}\sqrt{\Phi_{\sigma}(s)/\Phi_{\eps}(s)} = k$ for some $k\ge0$, uniformly for $\arg s \in [-\pi,\pi]$; \label{A5}
\item $\lim_{s\to0}\abs{s\sqrt{\Phi_{\sigma}(s)/\Phi_{\eps}(s)}} =0$. \label{A6}
\end{enumerate}

The assumption \ref{A5} stated here is slightly stronger than the corresponding one in \cite{KOZ2019}, where convergence only of the absolute value in the range $\arg s \in (\pi/2,\pi) \cup (-\pi,-\pi/2)$ was asked. We believe however that convergence of the function itself in the larger range is necessary for the proof of Theorem 3.5 of \cite{KOZ2019}.

In this paper, we will restrict the constitutive functions $\phi_{\sigma}$ and $\phi_{\eps}$ to positive Radon measures instead of general distributions, and we will derive a single practical condition (\eqref{eq:TD} below) which implies the above six technical conditions. Hence, the results from \cite{KOZ2019} are valid unconditionally in this setting.

\section{The distributed-order derivative with a measure as weight}\label{Sec: DoDmeasures}

Let $\mu(\alpha)$ be a positive Radon measure supported in $[0,1]$. With $\phi(\alpha) = \mu(\alpha)$, the DOFD \eqref{DefInt} of $u\in \MS'_{+}$ becomes the classical (vector-valued) integral
\begin{equation}\label{eq: DoDmeasures}
	D^{\mu} u(t) = \int_0^1 \fracD{0}{\alpha}{t}u(t) \dif \mu(\alpha). 
\end{equation}

The Laplace transform of the distributed-order fractional derivative given by \eqref{LTdofd} takes the form 
\[
	\ML \left( \int_0^1\fracD{0}{\alpha}{t}u(t) \dif \mu (\alpha) \right)(s) =  \Phi(s) \cdot \ML u(s), \quad \Re s >0,
\]
where we introduced the notation 
\[
	\Phi(s) := \int_{0}^{1}s^{\alpha} \dif \mu(\alpha).
\]

We first turn our attention to the analysis of the function $\Phi$, as it plays an essential role in later sections.

\subsection*{Analysis of $\Phi$}

We begin with a lemma giving bounds for the function $\Phi(s)$. Note that the lower bound implies in particular that $\Phi(s)$ has no zeros on $\C \setminus (-\infty,0]$.
\begin{lemma}
\label{lemma: 1} 
Let $s=R\e^{\I\theta}$ with $R>0$ and $\abs{\theta}<\pi$. Then 
\[
	\cos(\,\abs{\theta}/2)\min(1, R)\mu([0,1]) \le \abs{\Phi(s)} \le \max(1, R)\mu([0,1]).
\]
\end{lemma}
\begin{proof}
The second inequality is obvious. For the first one, write
\[
	\Phi(s) = \e^{\I\theta/2}\biggl(\int_{0}^{1}R^{\alpha}\cos(\alpha\theta-\theta/2)\dif \mu(\alpha) + \I \int_{0}^{1}R^{\alpha}\sin(\alpha\theta-\theta/2)\dif \mu(\alpha)\biggr).
\]
We have that $\abs{\Phi(s)} \ge \int_{0}^{1}R^{\alpha}\cos(\alpha\theta-\theta/2)\dif \mu(\alpha) \ge \cos(\theta/2)\min(1, R)\mu([0,1])$.
\end{proof}

The lemma implies the existence of a unique solution of the equation (\ref{DOE}) for a given $f\in \MS'_{+}$, with operator $D^\phi$ defined by \eqref{eq: DoDmeasures}. Indeed, taking Laplace transforms one obtains 
\[
	\ML u(s) = \frac{\ML f(s)}{\Phi(s)}, \qquad \Re s > 0.
\]
By Lemma \ref{lemma: 1} one can bound the right hand side as in \eqref{LTbounds}, yielding existence of the inverse Laplace transform and a solution.

It will be important to study the asymptotic behavior of $\Phi(s)$ when $s\to \infty$. This behavior depends only on the behavior of $\mu$ near the largest point of its support: 
\[
	M \coloneqq \max\supp \mu.
\]
This point is characterized by the property that for every $\epsilon>0$, $\mu([M-\epsilon,M])>0$ and $\mu((M,M+\epsilon))=0$. We have
\begin{equation}
\label{eq: behavior Phi}
	\Phi(s) = \mu(\{M\}) s^{M}+  o(\,\abs{s}^{M}), \quad s \to \infty.
\end{equation}
Indeed, by dominated convergence, 
\[
	\int_{[0,M)}s^{\alpha-M}\dif \mu(\alpha) \to 0, \quad \mbox{as $s\to\infty$}.
\]
Of course, when $\mu$ has no point mass at $M$, this just reduces to $\Phi(s) = o(\,\abs{s}^{M})$.
One the other hand, one also sees that for every $\epsilon>0$ 
\begin{equation}
\label{eq: lower bound Phi}
	\Phi(s) \gtrsim_{\epsilon} \abs{s}^{M-\epsilon}, \quad s \to \infty.
\end{equation}
This follows from
\[
	\Phi(s) = \int_{[0,M-\epsilon)}s^{\alpha}\dif \mu(\alpha) + \int_{[M-\epsilon, M]}s^{\alpha}\dif \mu(\alpha) = o(\,\abs{s}^{M-\epsilon}) + \int_{[M-\epsilon, M]}s^{\alpha}\dif \mu(\alpha) .
\]
As in the proof of Lemma \ref{lemma: 1}, we can bound the absolute value of the second integral from below by $\abs{s}^{M-\epsilon} \cos(\epsilon\theta/2) \mu([M-\epsilon,M])$. Here the measure of the interval $[M-\epsilon, M]$ is non-zero by definition of the support of a measure.

From the preceding discussion it is clear that two measures $\mu_{1}$ and $\mu_{2}$ with $\max \supp\mu_{1} = \max\supp\mu_{2} = M$ that coincide on a neighborhood $[M-\epsilon,M]$ of $M$, however small, display the same asymptotic behavior for large $s$, in the sense that $\Phi_{1}(s) \sim \Phi_{2}(s)$.

\subsubsection*{Examples} We close the section with some important examples. 

The simplest case is that of a linear combination of point measures. Let $0\le \alpha_{1} < \dotso < \alpha_{n}\le 1$, and suppose $a_{1}, \dotso, a_{n}$ are positive numbers. If $\mu(\alpha) = a_{1}\delta(\alpha-\alpha_{1}) + \dotsb + a_{n}\delta(\alpha-\alpha_{n})$, then
\[
	\Phi(s) = \sum_{i=1}^{n}a_{i}s^{\alpha_{i}}.
\]
Of course, $M=\alpha_{n}$ in this case, and $\Phi(s) \sim a_{n}s^{\alpha_{n}} = \mu(\{M\})s^{M}$ as $s\to\infty$.

A second important case is that of absolutely continuous measures $\dif \mu(\alpha) = f(\alpha)\dif \alpha$ for some $L^{1}$-function $f$. Let us consider for example an exponential law $f(\alpha) = \tau^{\alpha}$ (restricted to the interval $[0,1]$), where $\tau$ is a positive constant.
We then have $M=1$, 
\[
	\Phi(s) = \int_{0}^{1}s^{\alpha}\tau^{\alpha}\dif \alpha = \frac{\tau s - 1}{\log(\tau s)},
\]
and $\Phi(s) \sim \tau s/\log s$.

When the $L^{1}$-function $f$ is sufficiently regular near the maximum of its support, in the sense that it can be approximated by a power function $(M-\alpha)^{\kappa}$, the asymptotic behavior of $\Phi$ can be determined from the following lemma.
\begin{lemma}
Let $M\in (0, 1]$ and $\kappa > -1$. For $\abs{s}\to\infty$,
\[
	\int_{0}^{M}s^{\alpha}(M-\alpha)^{\kappa}\dif \alpha = \Gamma(\kappa+1)\frac{s^{M}}{(\log s)^{\kappa+1}} + O\biggl(\frac{1}{\log\abs{s}}\biggr).
\]
\end{lemma}
\begin{proof}
For $s$ with $\abs{s}>1$, set $\theta_{s} = \arg\log s = \arctan(\arg s/\log\abs{s}) \in (-\pi/2,\pi/2)$. Consider the contour $\Gamma_{1}\cup\Gamma_{2}$, where $\Gamma_{1}$ is the arc of the circle with center $M$ and radius $M$ which goes from $0$ to $M-M\e^{-\I\theta_{s}}$, and $\Gamma_{2}$ is the line segment $[M-M\e^{-\I\theta_{s}}, M]$. The integral over $[0,M]$ equals the integral over $\Gamma_{1}\cup\Gamma_{2}$. We have
\begin{align*}
	\int_{\Gamma_{1}}s^{z}(M-z)^{\kappa}\dif z 	&= \int_{0}^{\theta_{s}}\exp\bigl(M(1-\e^{-\I\theta})\log s\bigr)(M\e^{-\I\theta})^{\kappa}M\I\e^{-\I\theta}\dif\theta \\
										&\lesssim M^{\kappa+1}\abs{\theta_{s}}\exp\bigl(M\log\abs{s}(1-\cos\theta_{s})\bigr) \lesssim \frac{1}{\log\abs{s}},
\end{align*} 
since $\abs{\theta_{s}} \lesssim 1/\log\abs{s}$ and $\cos(\theta_{s})=1+O(1/\log^{2}\abs{s})$. For the integral over $\Gamma_{2}$, we use the parametrization $z=M - v/\log s$, with $v\in[0,M\abs{\log s}]$. We get
\begin{align*}
	\int_{\Gamma_{2}}s^{z}(M-z)^{\kappa}\dif z	&= \int_{0}^{M\,\abs{\log s}}s^{M-\frac{v}{\log s}}\biggl(\frac{v}{\log s}\biggr)^{\kappa}\frac{\dif v}{\log s} \\
										&= \frac{s^{M}}{(\log s)^{\kappa+1}}\int_{0}^{M\,\abs{\log s}}\e^{-v}v^{\kappa}\dif v \\
										&= \frac{s^{M}}{(\log s)^{\kappa+1}}\bigl(\Gamma(\kappa+1)+ O((M\abs{\log s})^{\kappa}\e^{-M\,\abs{\log s}})\bigr).
\end{align*}
\end{proof}
Notice also that the result of this lemma is in agreement with the previous observations: $\Phi(s) = o(\,\abs{s}^{M})$ and $\Phi(s) \gtrsim_{\epsilon} \abs{s}^{M-\epsilon}$ for every $\epsilon>0$, i.e. with formulas \eqref{eq: behavior Phi} and \eqref{eq: lower bound Phi}, and the remark in between.

\section{The constitutive law}\label{Sec: TD}
Let us now state the general form of the constitutive equation, linking the stress $\sigma$ and the strain $\eps$ of a viscoelastic material, using positive measures as constitutive ``functions''. We consider two positive Radon measures $\mus$ and $\mue$ supported in $[0,1]$, and propose the constitutive equation in the following form:
\begin{equation}
\label{eq: const. eq. measures}
	\int_{0}^{1}\fracD{0}{\alpha}{t}\sigma(x,t)\dif \mus(\alpha) = \int_{0}^{1}\fracD{0}{\alpha}{t}\eps(x,t)\dif \mue(\alpha).
\end{equation}
It is  clear that for $\phi_{\sigma}=\mus$ and $\phi_{\eps}=\mue$, equation \eqref{eq: const. equation} takes the form \eqref{eq: const. eq. measures}. It can be argued that this setting is more intuitive, since then the integrals in \eqref{eq: const. eq. measures} might be interpreted as (weighted) averages of the fractional derivatives ${}_{0}D_{t}^{\alpha}\sigma$ and ${}_{0}D_{t}^{\alpha}\eps$, while the interpretation of $\langle \phi(\alpha), {}_{0}D_{t}^{\alpha}f\rangle$ where $\phi$ is a general distribution is less clear. 

It is also clear that most concrete examples of viscoelasticity, see e.g.\ \cite{KOZ2019, Mainardi}, are contained in this definition. For example, the Zener constitutive equation corresponds to choice $\mu_{\sigma}(\alpha) = \delta(\alpha)+\tau\delta(\alpha-\alpha_{1})$ and $\mu_{\eps}(\alpha) = \delta(\alpha) + \delta(\alpha-\alpha_{1})$. Also, the so-called power type model, which uses exponential functions for constitutive functions and is proposed in \cite{Atanackovic2002} could be obtained from \eqref{eq: const. equation} for $\dif\mus(\alpha) = a^{\alpha}\dif \alpha$, $\dif\mue(\alpha) = b^{\alpha}\dif \alpha$, where $a$ and $b$ are positive constants. 

The Laplace transform of \eqref{eq: const. eq. measures} reads
\begin{equation}
\label{eq: Laplace transform const. eq. measures}
	\Phi_{\sigma}(s) \ML_{t}\{\sigma(x,t)\}(s)  = \Phi_{\eps}(s)\ML_{t}\{\eps(x,t)\}(s),
\end{equation}
where we have set
\[
	\Phi_{\sigma}(s) \coloneqq \int_{0}^{1}s^{\alpha}\dif\mus(\alpha), \quad \Phi_{\eps}(s) \coloneqq \int_{0}^{1}s^{\alpha}\dif\mue(\alpha).
\] 
These two functions play a vital role in the existence theory for the wave equation modeling wave propagation in viscoelastic media. 
\subsection*{Thermodynamical restriction}
Let us come back to the constitutive equation \eqref{eq: const. eq. measures}.
When modeling physical phenomena, the constitutive equation is subject to certain restrictions coming from energy considerations. These are usually referred to as thermodynamical restrictions. One defines the so-called complex modulus $\hat{E}$ via the relation
\[
	\F_{t}\{\sigma(x,t)\}(\omega) = \hat{E}(\omega)\F_{t}\{\eps(x,t)\}(\omega),
\] 
obtained by taking Fourier transforms in the constitutive equation. Note that $\hat{E}(\omega) = \Phi_{\eps}(\I\omega)/\Phi_{\sigma}(\I\omega)$ is a well-defined function for $\omega>0$. The functions $\Re \hat{E}(\omega)$ and $\Im \hat{E}(\omega)$ are called the storage modulus and loss modulus respectively. They can be related to the elastically stored energy and the dissipated energy respectively for sinusoidal excitations, and hence are required to be non-negative for every $\omega > 0$, see e.g.\ \cite[Chapter 9]{Tschoegl} or \cite[Section 2.7]{Mainardi} (see also \cite{BT1986} for the special case of the fractional Zener model). Hence we require
\[
	\forall \omega>0: \Re\hat{E}(\omega)\ge0, \quad \Im\hat{E}(\omega) \ge 0.
\] 
We have 
\[
	\Re\hat{E}(\omega) = \frac{1}{\abs{\Phi_{\sigma}(\I\omega)}^{2}}\iint_{[0,1]^{2}}\omega^{\alpha+\beta}\cos\Bigl(\frac{\pi}{2}(\alpha-\beta)\Bigr)\dif\mue(\alpha)\dif\mus(\beta),
\]
and 
\begin{align*}
	\Im\hat{E}(\omega)	&= \frac{1}{\abs{\Phi_{\sigma}(\I\omega)}^{2}}\iint_{[0,1]^{2}}\omega^{\alpha+\beta}\sin\Bigl(\frac{\pi}{2}(\alpha-\beta)\Bigr)\dif\mue(\alpha)\dif\mus(\beta)\\
					&= \frac{1}{\abs{\Phi_{\sigma}(\I\omega)}^{2}}\iint_{\Delta}\omega^{\alpha+\beta}\sin\Bigl(\frac{\pi}{2}(\alpha-\beta)\Bigr)\dif\left\{\mue(\alpha)\times\mus(\beta)-\mus(\alpha)\times\mue(\beta)\right\},
\end{align*}
where $\Delta$ is the triangle $\Delta = \{(\alpha,\beta)\in [0,1]: \alpha>\beta\}$.

The condition $\Re\hat{E}(\omega)\ge0$ readily follows from the fact that $\mus$ and $\mue$ are positive measures. A sufficient condition for the non-negativity of $\Im\hat{E}(\omega)$ is that $\mue\times\mus - \mus\times\mue$  is a non-negative measure on the triangle $\Delta$. This condition is not necessary, as can be seen from simple examples such as $\mus(\alpha) = \delta(\alpha-1/5) + \delta(\alpha-3/5)$ and $\mue(\alpha) = \delta(\alpha-2/5) + \delta(\alpha-4/5)$. However, the aforementioned sufficient condition appears quite naturally in the analysis of the existence of solutions of the wave equation. Hence, from now on we assume that the measures $\mus$ and $\mue$ appearing in the constitutive equation \eqref{eq: const. eq. measures} satisfy the thermodynamical restriction \eqref{eq:TD}:
\begin{equation}
\label{eq:TD}
\nues(\alpha, \beta) \coloneqq \mue(\alpha)\times\mus(\beta) - \mus(\alpha)\times\mue(\beta) \text{ is a non-negative measure on } \Delta. \tag{\textbf{T}}
\end{equation}

\subsubsection*{Special cases}Let us rewrite this condition for some specific forms of the measures.
Suppose that 
\begin{equation}
\label{eq: LFM}
	\mus(\alpha) = \sum_{i=1}^{n}a_{i}\delta(\alpha-\alpha_{i}), \quad \mue(\alpha) = \sum_{i=1}^{n}b_{i}\delta(\alpha-\alpha_{i}),
\end{equation}
where $0\le\alpha_{1} < \dotso < \alpha_{n} <1$, $a_{i}$, $b_{i}\ge0$, but not $a_{i}=b_{i}=0$ for some $i$. The property \eqref{eq:TD} then reduces to 
\begin{equation}
\label{eq: thermo LFM}
	\frac{a_{1}}{b_{1}} \ge \frac{a_{2}}{b_{2}} \ge \dotso \frac{a_{n}}{b_{n}} \ge 0,
\end{equation}
with the convention that $a/0=\infty$. This condition unifies the thermodynamical restrictions on the coefficients in the constitutive equations describing four types of viscoelastic materials proposed in the paper \cite{AKOZ}. 

In the case that $\dif\mus(\alpha)=f(\alpha)\dif\alpha$, $\dif\mue(\alpha)=g(\alpha)\dif\alpha$, with $f$, $g\in L^{1}$, the condition reduces to 
\[
	g(\alpha)f(\beta) \ge g(\beta)f(\alpha) \mbox{ a.e.\ on the set } \alpha > \beta.
\] 
for positive $g$, this is equivalent to: $f/g$ is a non-increasing function. In case of the exponential model, i.e. 
\begin{equation}
\label{eq: thermo explaw}
	\dif\mus(\alpha) = a^{\alpha}\dif \alpha, \quad \dif\mue(\alpha) = b^{\alpha}\dif \alpha,
\end{equation}
for $a$ and $b$  positive constants, the thermodynamical restriction  \eqref{eq:TD} becomes\footnote{The case $a=b$ is often excluded, since in that case the constitutive equation reduces to the trivial equation $\sigma=\eps$.} $a\le b$. This is again the proposed thermodynamical restriction in \cite{AKOZ}.\\

To conclude the section we introduce some terminology. By \emph{proper fractional models} we will mean those models for which the support of $\mus$ or $\mue$ intersects the open interval $(0,1)$, i.e.\ those models for which there appears a proper fractional differential operator $\fracD{0}{\alpha}{t}$, $0<\alpha<1$, in the constitutive equation \eqref{eq: const. eq. measures}. The remaining models are referred to as the \emph{classical mechanical models} \cite{Mainardi}, and are given by
\begin{equation}
\label{eq: classical models}
	\mus(\alpha) = a_{0}\delta(\alpha) + a_{1}\delta(\alpha-1), \quad \mue(\alpha) = b_{0}\delta(\alpha) + b_{1}\delta(\alpha-1),
\end{equation} 
where $a_{0}, a_{1}, b_{0}, b_{1}$ are non-negative constants. Bearing in mind the thermodynamical restriction \eqref{eq:TD}, which reduces here to $a_{0}/b_{0} \ge a_{1}/b_{1}$, we discern five cases.
\begin{itemize}
	\item the Hooke model: $a_{0}\sigma=b_{0}\eps$ with $a_{0},b_{0}>0$ (the models $a\od{\sigma}{t}=b\od{\eps}{t}$ and $a(\sigma+c\od{\sigma}{t})=b(\eps+c\od{\eps}{t})$ are also equivalent to the Hooke model);
	\item the Newton model: $a_{0}\sigma=b_{1}\od{\eps}{t}$ with $a_{0},b_{1}>0$;
	\item the Voigt model: $a_{0}\sigma=b_{0}\eps+b_{1}\od{\eps}{t}$ with $a_{0},b_{0},b_{1}>0$; 
	\item the Maxwell model: $a_{0}\sigma+a_{1}\od{\sigma}{t}=b\od{\eps}{t}$ with $a_{0},a_{1},b>0$;
	\item the Zener or Standard Linear Solid (SLS) model: $a_{0}\sigma+a_{1}\od{\sigma}{t}=b_{0}\eps+b_{1}\od{\eps}{t}$ with $a_{0},a_{1},b_{0},b_{1}>0$ and $a_{0}/b_{0} > a_{1}/b_{1}$.
\end{itemize}

\section{Existence and uniqueness of solutions}\label{Sec: existence and uniqueness} 
In this section we discuss existence and uniqueness of solutions for the distributed-order fractional wave equation \eqref{eq: DFWE} obtained from the system of the equilibrium equation \eqref{eq: NL}, the constitutive law \eqref{eq: const. eq. measures}, and the strain measure \eqref{eq: sm}. 
In what follows we will show that the conditions \ref{A1}- \ref{A6} given in Subsection \ref{Sec: Six conditions} hold in the case of measures satisfying \eqref{eq:TD}, implying the existence and uniqueness result, as well as a representation formula and support properties, by the results from \cite[Section 3]{KOZ2019}. 

To start, let us note that from Lemma \ref{lemma: 1} it follows that for $\Re s>0$:
\[
	\frac{\Phi_{\eps}(s)}{\Phi_{\sigma}(s)} \lesssim \frac{\max\{1,\abs{s}\}}{\min\{1,\abs{s}\}} \le \frac{(1+\abs{s})^2}{\Re s}, \quad \text{and}\quad \frac{\Phi_{\sigma}(s)}{\Phi_{\eps}(s)} \lesssim \frac{\max\{1,\abs{s}\}}{\min\{1,\abs{s}\}} \le \frac{(1+\abs{s})^2}{\Re s}.
 \] 
This implies in particular that both $\Phi_{\eps}(s)/\Phi_{\sigma}(s)$ and $\Phi_{\sigma}(s)/\Phi_{\eps}(s)$ are Laplace transforms of tempered distributions in $\MS'_{+}$, thus \ref{A1} and \ref{A3} hold. From assumption \ref{A1} we can derive the wave equation  \eqref{eq: DFWE}. Indeed, by the Laplace inversion formula \eqref{eq: Laplace inverse formula} and denoting 
\begin{equation}\label{eq: L}
	L(t): = \ML^{-1}\biggl(\frac{\Phi_{\eps}(s)}{\Phi_{\sigma}(s)}\biggr)(t).
\end{equation}
 one can solve the constitutive equation \eqref{eq: const. eq. measures} for $\sigma$ as:
\[
	\sigma = L(t)\ast_{t}\eps.
\]
Taking partial derivative with respect to $x$ of both sides, and using equations \eqref{eq: NL} and \eqref{eq: sm} to write $\partial_{x}\sigma$ and $\eps$ in terms of the displacement $u$, we obtain the distributed-order fractional wave equation \eqref{eq: DFWE}, i.e. 
\begin{equation}
\label{eq: DFWE1}
	\dpd[2]{}{t}u(x,t) = \ML^{-1}\biggl(\frac{\Phi_{\eps}(s)}{\Phi_{\sigma}(s)}\biggr)(t)\ast_{t}\dpd[2]{}{x}u(x,t), \qquad x\in\R, \quad t>0.
\end{equation}

Theorem 3.2 of \cite{KOZ2019} states that one has existence and uniqueness of solutions under the assumptions \ref{A2} and \ref{A3}. The unique solution to the Cauchy problem \eqref{eq: DFWE1}, $u(x,0) = u_{0}(x)$ and $\partial_{t}u(x,0)=v_{0}(x)$ is then given by
\[
	u(x,t) = S(x,t) \ast(u_{0}(x) \delta'(t) + v_{0}(x)\delta(t)),
\]
where
\begin{equation}
\label{eq: definition S}
	S(x,t) = \ML^{-1}_{s\to t}\{\tilde{S}(x, s)\}, \quad \tilde{S}(x,s) = \frac{\Psi(s)}{2s}\exp\bigl(-\abs{x}s\Psi(s)\bigr),
\end{equation}
and where, here and throughout the paper, $\Psi$ is defined  as
\begin{equation}\label{Psi}
	\Psi(s) := \sqrt{\frac{\Phi_{\sigma}(s)}{\Phi_{\eps}(s)}}.
\end{equation}
The formula for the fundamental solution $S$ can be found by taking Laplace transforms with respect to $t$ in $\partial_{t}^{2}u - L(t)\ast_{t}\partial_{x}^{2}u = u_{0}(x)\delta'(t) + v_{0}(x)\delta(t)$, and then solving the ODE in $x$.
The assumptions \ref{A2} and \ref{A3} imply that for each $s$ with $\Re s>0$, the found solution $\tilde{u}(x,s) = \tilde{S}(x,s)\ast_{x}(su_{0}(x) + v_{0}(x))$ is tempered in $x$, and that $S$ is well-defined as an element of $\MS'(\R\times\R_{+})$.

We now show that \ref{A2} follows from the thermodynamical restriction \eqref{eq:TD}. Condition \ref{A2} is equivalent to $s^{2} + \xi^{2}\Phi_{\eps}(s)/\Phi_{\sigma}(s) \neq 0$ for $\Re s>0$ and for every $\xi\in\R$. Write $s=R\e^{\I\theta}$ with $\theta\in(-\pi/2,\pi/2)$. If $\theta=0$, then $\Phi_{\eps}(s)$ and $\Phi_{\sigma}(s)$ are real and positive, so that $R^{2} + \xi^{2}\Phi_{\eps}(R)/\Phi_{\sigma}(R) > 0$. Suppose now that $\theta \in (0,\pi/2)$. Then
\[
	\Im\frac{\Phi_{\eps}(s)}{\Phi_{\sigma}(s)} 
		= \frac{1}{\abs{\Phi_{\sigma}(s)}^{2}}\iint_{\Delta}R^{\alpha+\beta}\sin((\alpha-\beta)\theta)\dif\left\{\mue(\alpha)\times\mus(\beta)-\mus(\alpha)\times\mue(\beta)\right\},
\]
By \eqref{eq:TD}, this imaginary part is $\ge0$, so that $\Im\bigl(s^{2} + \xi^{2}\Phi_{\eps}(s)/\Phi_{\sigma}(s)\bigr) > 0$. Similarly one shows that $\Im\bigl(s^{2} + \xi^{2}\Phi_{\eps}(s)/\Phi_{\sigma}(s)\bigr) < 0$ if $\theta \in (-\pi/2,0)$.

Having proven the assumptions \ref{A1}, \ref{A2}, and \ref{A3}, we can conclude that the following theorem holds.
\begin{theorem}
\label{th: existence and uniqueness}
Let $\mus$ and $\mue$ be positive Radon measures on $[0,1]$, satisfying the thermodynamical restriction \eqref{eq:TD}. Then the Cauchy problem for the distributed-order fractional wave equation \eqref{eq: DFWE} with initial conditions
\begin{equation}
\label{eq: Cauchy data}
	u(x,0) = u_{0}(x), \quad \dpd{}{t}u(x,0) = v_{0}(x),
\end{equation}
where $u_{0}, v_{0}\in \MS'(\R)$, has a unique solution $u\in \MS'(\R\times\R_{+})$.
\end{theorem}

We now turn our attention to the assumptions \ref{A4}, \ref{A5}, and \ref{A6}. These assumptions imply that the fundamental solution is supported in the cone $\abs{x} \le ct$, with $c=1/k$, ($k$ coming from \ref{A5}), and yield a representation of the fundamental solution inside this cone in terms of an absolutely convergent integral.

Recall from the end of the Section \ref{Sec: TD} that by proper fractional models we mean constitutive models described via \eqref{eq: const. eq. measures} for which the support of $\mus$ or $\mue$ intersects the open interval $(0,1)$. 

\begin{lemma}
\label{lem: condition A4}
For proper fractional models satisfying \eqref{eq:TD}, the functions $\Phi_{\sigma}(s)$ and $\Phi_{\eps}(s)$ are non-zero for $\arg s \in [-\pi,\pi]$, $s\neq0$. In particular, the condition \ref{A4} holds.
\end{lemma}
\begin{proof}
By Lemma \ref{lemma: 1}, we see that $\Phi_{\sigma}(s)$ and $\Phi_{\eps}(s)$ have no zeros in the range $\arg s \in (-\pi,\pi)$, $s\neq0$. Suppose now that $s=R\e^{\I\pi}$ with $R>0$ is a zero of $\Phi_{\sigma}$. Then
\[
	0 = \Im\Phi_{\sigma}(s) = \Im\biggl( \mus(\{0\}) + \int_{(0,1)}s^{\alpha}\dif\mus(\alpha) - R \mus(\{1\})\biggr) = \int_{(0,1)}R^{\alpha}\sin(\alpha\pi)\dif\mus(\alpha),
\]
 and so $\mus((0,1))=0$, since $\sin(\alpha\pi)>0$ on this interval. In a similar fashion we conclude that $\mus((0,1))=0$ if $\Phi_{\sigma}$ has a zero of the form $R\e^{-\I\pi}$. Hence, the occurrence of a zero would imply that $\mus$ is of the form $a_{0}\delta(\alpha) + a_{1}\delta(\alpha-1)$ with positive $a_{0}, a_{1}$. But then \eqref{eq:TD} implies that we are in the case \eqref{eq: classical models}. Indeed, applying \eqref{eq:TD} to the set $\{1\}\times(0,1) \subseteq \Delta$ yields
\[
 	\nues(\{1\}\times(0,1)) = -a_{1}\mue((0,1)) \ge 0,
\]
which is only possible if $\mue((0,1))=0$.

If $\Phi_{\eps}$ has a zero of the form $s = R\e^{\pm\I\pi}$, $R>0$, then we can conclude in a similar fashion that $\mue(\alpha) = b_{0}\delta(\alpha) + b_{1}\delta(\alpha-1)$ for positive $b_{0},b_{1}$. Applying \eqref{eq:TD} to the set $(0,1)\times\{0\}\subseteq\Delta$, we see that $\mus((0,1))=0$, so that we are again in the case \eqref{eq: classical models}.
\end{proof}
For the classical mechanical models, both the Hooke and the Newton model satisfy \ref{A4}, but the other models have different branch points:  $-b_{0}/b_{1}$ and $\infty$ in the Voigt model, $0$ and $-a_{0}/a_{1}$ in the Maxwell model, and $-a_{0}/a_{1}$ and $-b_{0}/b_{1}$ in the Zener model. The previous lemma shows that these are the only models where such different branch points occur. However, the conclusion of Theorem \ref{th: support S} below can be readily adapted to these models. 

\bigskip

We now show \ref{A5} for all models satisfying \eqref{eq:TD} (including the models \eqref{eq: classical models}). We start with an observation concerning the maximal points of the support of $\mus$ and $\mue$: 
\[
	\max\supp \mus \eqqcolon M_{\sigma} \le M_{\eps} \coloneqq \max\supp\mue.
\]
Indeed, suppose to the contrary that $M_{\sigma} > M_{\eps}$. Let $\epsilon>0$ be so small that $M_{\sigma}-\epsilon>M_{\eps}$. Applying \eqref{eq:TD} to the set $[M_{\sigma}-\epsilon,M_{\sigma}] \times [0, M_{\eps}] \subseteq \Delta$ gives
\[
	0 \le \nues([M_{\sigma}-\epsilon,M_{\sigma}] \times [0, M_{\eps}]) = -\mus([M_{\sigma}-\epsilon,M_{\sigma}])\mue([0, M_{\eps}]) < 0,
\]
a contradiction.

If $M_{\eps} > M_{\sigma}$, then \ref{A5} follows with $k=0$. Indeed, by \eqref{eq: behavior Phi} and \eqref{eq: lower bound Phi} we have that $\Phi_{\sigma}(s) \lesssim \abs{s}^{M_{\sigma}}$, while $\Phi_{\eps}(s) \gtrsim \abs{s}^{M_{\eps}-\delta}$, for every $\delta>0$.

If $M_{\eps} = M_{\sigma} = M$, there are two possibilities. Either $\mus(\{M\})>0$ and $\mue(\{M\})>0$, or $\mus(\{M\}) = 0$. The case $\mus(\{M\})>0$ and $\mue(\{M\}) = 0$ again leads to a contradiction by considering the set $V = \{M\} \times [0,M)$ for which one would have $\nues(V)<0$. In the case that $\mus(\{M\})>0$ and $\mue(\{M\})>0$, \ref{A5} follows easily, since by \eqref{eq: behavior Phi} we have 
\[
	\frac{\Phi_{\sigma}(s)}{\Phi_{\eps}(s)} = \frac{\mus(\{M\})s^{M}+o(\,\abs{s}^{M})}{\mue(\{M\})s^{M}+o(\,\abs{s}^{M})} = \frac{\mus(\{M\})}{\mue(\{M\})} + o(1), \quad \text{as } s\to\infty.
\]
Hence, in this case \ref{A5} holds with $k=\sqrt{\mus(\{M\})/\mue(\{M\})}$. The case $\mus(\{M\})=0$ is more subtle. We need the following function associated to the measures $\mus$ and $\mue$:
\begin{equation}
\label{eq: def F}
	F(x) \coloneqq \frac{\mus([x, M])}{\mue([x,M])}, \quad x\in [0,M].
\end{equation}
The function $F$ is non-negative and non-increasing. Indeed, let $0\le y<x\le M$. Applying \eqref{eq:TD} to the set $[x,M]\times[y,x)$ yields
\[
	\mue([x,M])\mus([y,x)) \ge \mus([x,M])\mue([y,x)).
\]
Writing $\mus([y,x)) = \mus([y,M])-\mus([x,M])$, $\mue([y,x)) = \mue([y,M])-\mue([x,M])$, we get that 
\[
	\mue([x,M])\mus([y,M]) \ge \mus([x,M])\mue([y,M]),
\]
so $F(y)\ge F(x)$. Since $F$ is non-increasing, it has a limit:
\[
	\tau \coloneqq \lim_{x\to M}F(x) \ge 0.
\]
\begin{lemma}
Suppose that $M_{\sigma}=M_{\eps} = M$ and that $\mus(\{M\})=\mue(\{M\})=0$. For every $\epsilon>0$ there exists a $\delta>0$ such that for every Borel set $A\subseteq [M-\delta, M]$, 
\[
	\tau\mue(A) \le \mus(A) \le (\tau+\epsilon)\mue(A).
\]
\end{lemma}
\begin{proof}
First note that $M>0$ since we assume that the measures $\mus$, $\mue$ are non-zero. Given $\epsilon>0$, let $z$ be such that $x\ge z \implies F(x) \le \tau+ \epsilon/2$, and consider the function
\[
	G_{z}(x) \coloneqq \frac{\mus([z,x))}{\mue([z,x))} = \frac{\mus([z,M]) - \mus([x,M])}{\mue([z,M]) - \mue([x,M])}.
\]
This function is well defined for $x$ sufficiently close to $M$, since $\mue([z,M])>0$ and 
$$\lim_{x\to M}\mue([x,M])=\mue(\{M\})=0$$
 by continuity of the measure $\mue$. Also $\lim_{x\to M}\mus([x,M]) = \mus(\{M\})=0$, so that 
\[
	\lim_{x\to M}G_{z}(x) = \frac{\mus([z,M])}{\mue([z,M])} = F(z).
\]
Let now $\delta_{1}>0$ be such that $G_{z}(x) \le F(z) + \epsilon/2$, for $x\ge M-\delta_{1}$. 

Suppose first that $A$ is of the form $[y,x)$ with $M-\delta_{1}\le y < x \le M$. The inequality $\tau \mue([y,x)) \le\mus([y,x))$ follows from $\mue([x,M])\mus([y,x)) \ge \mus([x,M])\mue([y,x))$ (apply \eqref{eq:TD} to $[x,M]\times[y,x) \subseteq \Delta$) and the fact that $F$ is non-increasing. For this inequality it is actually not required that $y\ge M-\delta_{1}$. For the second inequality we apply \eqref{eq:TD} to $[y,x)\times[z,y) \subseteq \Delta$ to get
\begin{align*}
	\mus([y,x))	&\le \frac{\mus([z,y))}{\mue([z,y))}\mue([y,x)) = G_{z}(y)\mue([y,x)) \\
				&\le (F(z) + \epsilon/2)\mue([y,x)) \le (\tau+\epsilon)\mue([y,x)). 
\end{align*}

suppose now that $A\subseteq (M-\delta_{1},M]$ is an arbitrary Borel set. Since $\mus(\{M\})=\mue(\{M\})=0$ we may actually assume that $A \subseteq (M-\delta_{1}, M)$. If $A$ is the open interval $(x,y)$, then the conclusion of the lemma follows upon writing $(x,y)=\cup_{n}[x+1/n, y)$ and using the continuity property of measures. If $A$ is an arbitrary open set, then $A$ is a countable disjoint union of open intervals, and the conclusion holds by additivity of the measures. The general case finally follows from the fact that $\mus$ and $\mue$ are outer-regular (they are Radon measures). Hence $\mus(A) = \inf\{\mus(G): G \supseteq A, G \mbox{ open}\}$, and similarly for $\mue$.

We conclude that the desired inequalities hold for every Borel set $A\subseteq (M-\delta_{1}, M]$, so selecting any $\delta<\delta_{1}$, the inequalities also hold for every Borel set $A\subseteq[M-\delta, M]$, and the lemma is proven.
\end{proof}

We are now able to the following prove proposition yielding \ref{A5}.
\begin{proposition}
\label{prop: existence tau}
There exists a number $\tau \ge 0$ so that, uniformly for $\theta\in[-\pi,\pi]$, 
\[
	\frac{\Phi_{\sigma}(R\e^{\I\theta})}{\Phi_{\eps}(R\e^{\I\theta})} \to \tau, \quad \text{ as }R\to\infty.
\]
\end{proposition}
\begin{proof}
We may assume that $M_{\sigma}=M_{\eps}=M$, since in the case $M_{\sigma}<M_{\eps}$ the proposition holds with $\tau=0$. We may also assume that $\mus(\{M\})=\mue(\{M\})=0$. If both of these numbers are positive, then the proposition holds with $\tau=\mus(\{M\})/\mue(\{M\})$, and if $\mus(\{M\})=0$, $\mue(\{M\})>0$, it holds with $\tau=0$. As we remarked before, the case $\mus(\{M\})>0$, $\mue(\{M\})=0$ is not allowed by \eqref{eq:TD}.

Let now $\epsilon>0$ be arbitrary, and select a corresponding $\delta>0$ according to the previous lemma. If $f$ is a real-valued function which is of fixed sign on $[M-\delta, M]$, then the lemma implies that
\[
	\abs{\int_{[M-\delta,M]}f(\alpha)\dif\left\{\mus(\alpha)-\tau\mue(\alpha)\right\}} \le \epsilon \abs{\int_{[M-\delta,M]}f(\alpha)\dif\mue(\alpha)}.
\]
From this it follows that 
\[
	\abs{\int_{[M-\delta,M]}s^{\alpha}\dif\left\{\mus(\alpha)-\tau\mue(\alpha)\right\}} \le 2\epsilon \abs{\int_{[M-\delta,M]}s^{\alpha}\dif\mue(\alpha)}.
\]
Indeed, if $\abs{M\theta} \le \pi/2$ or $3\pi/4 \le \abs{M\theta} \le \pi$, then $\cos(\alpha\theta)$ and $\sin(\alpha\theta)$ are of fixed sign on the interval $\alpha\in [M-\delta, M]$ (provided that $\delta<1/4$ say, which we may assume). Hence the claim follows by splitting the integrand $s^{\alpha}$ in real and imaginary parts. If $M\theta\in(\pi/2,3\pi/4)$, write $s^{\alpha} = \e^{\I\pi/4}\bigl(\cos(\alpha\theta-\pi/4) + \I\sin(\alpha\theta-\pi/4)\bigr)$. The functions $\cos(\alpha\theta-\pi/4)$ and $\sin(\alpha\theta-\pi/4)$ are of fixed sign in the interval $\alpha\in[M-\delta, M]$, so we again reach the same conclusion. Similarly if $M\theta\in(-3\pi/4,-\pi/2)$. Hence we get
\begin{align*}
\Phi_{\sigma}(s)		&= \int_{[M-\delta,M]}s^{\alpha}\dif\mus(\alpha) + o(\,\abs{s}^{M-\delta}) \\
					&=\tau\int_{[M-\delta,M]}s^{\alpha}\dif\mue(\alpha) +  \int_{[M-\delta,M]}s^{\alpha}\dif\left\{\mus(\alpha)-\tau\mue(\alpha)\right\} 
						+ o(\,\abs{s}^{M-\delta}) \\
					&=\tau\int_{[M-\delta,M]}s^{\alpha}\dif\mue(\alpha) + \zeta \abs{\int_{[M-\delta,M]}s^{\alpha}\dif\mue(\alpha)} + o(\,\abs{s}^{M-\delta}), 
\end{align*}
where $\zeta$ is some complex number with $\abs{\zeta}\le 2\epsilon$. Now since $\Phi_{\eps}(s) \gtrsim \abs{s}^{M-\delta}$ and $\Phi_{\eps}(s) = \int_{[M-\delta,M]}s^{\alpha}\dif\mue(\alpha) + o(\,\abs{s}^{M-\delta})$, we have
\[
	\frac{\Phi_{\sigma}(s)}{\Phi_{\eps}(s)} = \frac{\tau+\zeta'+o(1)}{1+o(1)},
\]
with $\abs{\zeta'} \le 2\epsilon$. 
From this it follows that 
\[
	\limsup_{\abs{s}\to\infty}\abs{\frac{\Phi_{\sigma}(s)}{\Phi_{\eps}(s)}-\tau} \le 2\epsilon.
\]
But $\epsilon>0$ was arbitrary, so $\lim_{\,\abs{s}\to\infty}\Phi_{\sigma}(s)/\Phi_{\eps}(s) = \tau$.
\end{proof}

Finally, we show \ref{A6}, which is easy. In fact, we have the stronger estimate 
\begin{equation}
\label{eq: bound near 0}
	\Psi(s) \lesssim \frac{1}{\sqrt{\,\abs{s}}}, \quad \text{as $s\to0$},
\end{equation}
where $\Psi$ is as in \eqref{Psi}. For $\theta=\arg s \in [-3\pi/4, 3\pi/4]$ say, this follows from Lemma \ref{lemma: 1}. For values of $\theta$ close to $\pm\pi$, we need a slightly refined lower bound for $\Phi_{\eps}(r\e^{\I\theta})$ for small $r$. First of all we may assume that $\mue([0,x])>0$ for some $x < M_{\eps}$. If this is not the case, then $\mue(\alpha)=a\delta(\alpha-M_{\eps})$ for some $a>0$, and $\Phi_{\eps}(r\e^{\I\theta}) = ar^{M_{\eps}}\e^{\I M_{\eps}\theta} \gtrsim r^{M_{\eps}}$ as $r\to0$. If $\mue([0,x])>0$, then like in the proof of Lemma \ref{lemma: 1}, we have
\begin{align*}
	\abs{\Phi_{\eps}(r\e^{\I\theta})} 	&= \abs{\e^{\I\theta x/2}\int_{0}^{1}r^{\alpha}\e^{\I(\alpha\theta-\theta x/2)}\dif\mue(\alpha)}\\
							&\ge \int_{0}^{1}r^{\alpha}\cos(\alpha\theta-\theta x/2)\dif\mue(\alpha) = \int_{[0,x]}r^{\alpha} \cos(\alpha\theta-\theta x/2)\dif\mue(\alpha) + o(r^{x}) \\
							&\ge r^{x}\cos(\theta x/2)\mue([0,x]) + o(r^{x}) \quad \text{as $r\to 0$}.
\end{align*}
 If $\abs{\theta}\le\pi$, then $\cos(\theta x/2) > 0$, since $x < 1$. Hence $\abs{\Phi_{\eps}(r\e^{\I\theta})} \gtrsim r^{x}$ as $r\to0$.
 
Finally, we state a  theorem on support properties of the fundamental solution $S$ and its representation formula, which is an immediate consequence of the preceding discussion showing that the assumptions \ref{A4}, \ref{A5}, and \ref{A6} are satisfied, and \cite[Theorem 3.5]{KOZ2019}. 
\begin{theorem}
\label{th: support S}
Let $\mus$ and $\mue$ be two positive Radon measures on $[0,1]$, satisfying \eqref{eq:TD}, and forming a proper fractional model (i.e.\ they are not of the form \eqref{eq: classical models}). The fundamental solution $S$ (given by \eqref{eq: definition S}) of the distributed-order fractional wave equation \eqref{eq: DFWE} is supported in the forward cone $\abs{x}\le ct$, with $c=1/k$, $k$ as in \ref{A5}. Inside this cone, we have the following representation of $K=\partial_{t}S$ by an absolutely convergent integral:
\begin{equation}
\label{eq: K in cone}
	K(x,t) = \frac{1}{4\pi\I}\int_{0}^{\infty}\bigl(\Psi(q\e^{-\I\pi})\e^{\,\abs{x}q\Psi(q\e^{-\I\pi})}
			- \Psi(q\e^{\I\pi})\e^{\,\abs{x}q\Psi(q\e^{\I\pi})}\bigr)\e^{-qt}\dif q, \quad \abs{x}<ct,
\end{equation}
where $\Psi$ is defined in \eqref{Psi}.
\end{theorem}

\begin{remark}
Note that in the case $k=0$, the statement about the support of $S$ reduces to the trivial statement that $S$ is supported in the set $t\ge0$.
\end{remark}

For convenience of the reader, let us briefly sketch the proof idea of the theorem. One starts with the Laplace inversion formula
\[
	S(x,t) = \lim_{Y\to\infty}\frac{1}{2\pi\I}\int_{a-\I Y}^{a+\I Y} \tilde{S}(x,s)\e^{ts}\dif s = 
	\lim_{Y\to\infty}\frac{1}{2\pi\I}\int_{a-\I Y}^{a+\I Y}\frac{\Psi(s)}{2s}\exp\bigl(-\abs{x}s\Psi(s)+ts\bigr)\dif s,
\]
where $a$ is some positive number.
By Proposition \ref{prop: existence tau}, the argument of the exponential is asymptotic to $(t-k\abs{x})s$ for large $s$. When $\abs{x} > ct=t/k$, one uses Cauchy's formula to write this integral as an integral over a semicircle of radius $Y$ in the right half plane. When $Y\to\infty$, this integral converges to $0$ in view of the exponential decay.

When $\abs{x} < ct$, one goes to the left half plane. First one moves via a quarter circle from the point $a+\I Y$ to $a-Y = a+\e^{\I\pi}Y$, then via a Hankel contour one goes up to the origin and back on the other side of the brach cut to the point $a+\e^{-\I\pi}Y$, and finally via a quarter circle one moves to $a-\I Y$. The integrals over the quarter circles again vanish in the limit $Y\to\infty$. For the pieces with $\Re s<0$, this follows from the exponential decay of $\e^{(t-k\,\abs{x})s}$. The integral over the pieces with $0\le \Re s \le a$ (which are of bounded length) tends to zero in view of the bound $\tilde{S}(s) \lesssim 1/\abs{s}$. The remaining integral over the Hankel contour is
\begin{multline}
\label{eq: S Hankel}
	S(x,t) = \frac{1}{4\pi}\int_{-\pi}^{\pi}\Psi(\epsilon\e^{\I\theta})\exp\bigl(-\abs{x}\epsilon\e^{\I\theta}\Psi(\epsilon\e^{\I\theta})+t\epsilon\e^{\I\theta}\bigr)\dif \theta \\
			+ \frac{1}{4\pi\I}\int_{\epsilon}^{\infty}\bigl(\Psi(q\e^{\I\pi})\e^{\,\abs{x}q\Psi(q\e^{\I\pi})} - \Psi(q\e^{-\I\pi})\e^{\,\abs{x}q\Psi(q\e^{-\I\pi})}\bigr)\frac{\e^{-qt}}{q}\dif q, \quad \abs{x}<ct,
\end{multline}
for arbitrary $\epsilon>0$. After differentiating with respect to $t$, we can safely let $\epsilon\to0$ by the estimate \eqref{eq: bound near 0}. This yields \eqref{eq: K in cone}.
\begin{remark}
\begin{itemize}
	\item The theorem is stated in terms of $K=\partial_{t}S$ instead of $S$, since this avoids the singularity of $\tilde{S}$ at $s=0$. The solution of the Cauchy problem \eqref{eq: DFWE}, \eqref{eq: Cauchy data} can be expressed in term of $K$ as $u(x,t) = K(x,t)\ast(u_{0}(x)\delta(t) + v_{0}(x)H(t))$, where $H$ denotes the Heaviside function.
	\item For the classical models \eqref{eq: classical models}, the theorem also holds, with the exception that the integral in \eqref{eq: K in cone} now ranges between (the negatives of) the two branch points. In the Voigt model for example, we have
	\begin{gather*}
		K(x,t) = \frac{1}{4\pi\I}\int_{b_{0}/b_{1}}^{\infty}\bigl(\Psi(q\e^{-\I\pi})\e^{\,\abs{x}q\Psi(q\e^{-\I\pi})}
			- \Psi(q\e^{\I\pi})\e^{\,\abs{x}q\Psi(q\e^{\I\pi})}\bigr)\e^{-qt}\dif q,\quad t>0, \\
			\text{with} \quad \Psi(s) = \sqrt{\frac{a_{0}}{b_{0}+b_{1}s}}.
	\end{gather*}
\end{itemize}
\end{remark}


\section{Smoothness of the fundamental solution}\label{Sec: Smooth}

From the representation \eqref{eq: K in cone}, it follows that $K$ is smooth on the set $\{(x,t): 0<\abs{x}<ct\}$. In fact, it is even real analytic there. Indeed, the integral for $K$ and its derivatives still converge absolutely if one replaces $x$ and $t$ by $x+z_{1}$ and $t+z_{2}$, $z_{1}, z_{2}\in\C$ with $\abs{z_{1}}$ and $\abs{z_{2}}$ sufficiently small, for $x$ and $t$ with $0<\abs{x}<ct$. The real analyticity also holds for the classical models. However, in case of the Voigt or Zener model the analogue of the representation \eqref{eq: K in cone} cannot be used, since differentiating the integrand with respect to $x$ creates a singularity at $q=b_{0}/b_{1}$ which is not integrable. For those models, one uses an analogue of \eqref{eq: S Hankel} to see that they are real analytic there. For the Voigt model one uses for example the representation
\[
	K(x,t) = \frac{1}{4\pi}\int_{\mathcal H}\Psi(s)\e^{-\,\abs{x}s\Psi(s)+ts}\dif s, \quad \Psi(s) = \sqrt{\frac{a_{0}}{b_{0}+b_{1}s}},
\]
where $\mathcal H$ is a Hankel contour encircling the branch cut $(-\infty, -b_{0}/b_{1}]$.

For $\abs{x} > ct$, $K$ is also real analytic since it is zero there. The question remains whether one has smoothness on the boundary of the cone. We will show that this is almost always the case. The starting point is the representation of $K$ as inverse Laplace transform:
\begin{equation}
\label{eq: inverse Laplace K}
	K(x,t) = \lim_{Y\to\infty}\frac{1}{4\pi\I}\int_{a-\I Y}^{a+\I Y}\Psi(s)\exp\bigl(-\abs{x}s\Psi(s)+ts\bigr)\dif s,
\end{equation}
for $a>0$, provided this limit exists. Writing $s=a+\I y$, we see that integrand is bounded by 
\begin{equation}
\label{eq: bound inverse Laplace K}
	\exp\bigl(\,\abs{x}y\Im \Psi(a+\I y) + O(1)\bigr).
\end{equation}
By Proposition \ref{prop: existence tau}, we know that the imaginary part in this exponential tends to zero as $y\to \pm\infty$. We will show that it has opposite sign to $y$, and that it does not decay to zero too quickly. For $x\neq0$, we will conclude that the integrand decays exponentially, so that the integral converges absolutely, and that we may furthermore differentiate indefinitely with respect to $x$ and $t$. Before showing this in full generality, it is instructive to work out some special cases.

\subsubsection*{Special cases} We consider the cases discussed in \cite{KOZ2019} (and in Subsection \ref{Sec: TD}), namely the linear fractional model \eqref{eq: LFM} with the thermodynamical restriction in the form of \eqref{eq: thermo LFM}, and the exponential-type distributed-order model given by \eqref{eq: thermo explaw} with $a=\tau$ and $b=1$ with thermodynamical restriction $0<\tau<1$. 
 From the linear fractional model we exclude the case where $a_{i}/b_{i} = a_{j}/b_{j}$ for every $i, j$, since this case reduces to the classical wave equation. We have 
\[
	\frac{\Phi_{\sigma}(s)}{\Phi_{\eps}(s)} = \frac{\sum_{i=1}^{n}a_{i}s^{\alpha_{i}}}{\sum_{i=1}^{n}b_{i}s^{\alpha_{i}}}
\]
in the model \eqref{eq: LFM}, and 
\[
	\frac{\Phi_{\sigma}(s)}{\Phi_{\eps}(s)} = \frac{(\tau s-1)\log s}{(s-1)\log(\tau s)}
\]
in the exponential-type distributed-order model.

Suppose first that $a_{n} \neq 0$ and $b_{n}\neq 0$. Let $m$ be such that $\frac{a_{m}}{b_{m}} > \frac{a_{m+1}}{b_{m+1}} = \dotso = \frac{a_{n}}{b_{n}} \eqqcolon \tau$. Using Taylor approximations, we see that for large $s$
\[
	 \Psi(s) = \sqrt{\tau}\biggl\{1 + \frac{1}{2}\biggl(\frac{a_{m}}{a_{n}} - \frac{b_{m}}{b_{n}}\biggr)s^{\alpha_{m}-\alpha_{n}} + O\bigl(\,\abs{s}^{\alpha_{m-1} - \alpha_{n}} + \abs{s}^{2(\alpha_{m}-\alpha_{n})} + \abs{s}^{\alpha_{m} + \alpha_{n-1}-2\alpha_{n}})\bigr)\biggr\},
\]
and where the first error term only occurs of $m>1$. Hence 
\begin{equation}
\label{eq: example1}
	\Im \Psi(a+\I y) \sim \mp\frac{\sqrt{\tau}}{2}\sin\bigl((\alpha_{n}-\alpha_{m})\pi/2\bigr)\biggl(\frac{a_{m}}{a_{n}}-\frac{b_{m}}{b_{n}}\biggr)\abs{y}^{-(\alpha_{n}-\alpha_{m})}, \quad \text{as $y\to\pm\infty$}.
\end{equation}
In the case that $a_{m+1} = \dotso = a_{n} = 0$ and $a_{m}\neq 0$, then 
\[
	 \Psi(s) = \sqrt{\frac{a_{m}}{b_{n}}}s^{(\alpha_{m}-\alpha_{n})/2}\bigl\{1+O\bigl(\, \abs{s}^{\alpha_{n-1}-\alpha_{n}} + \abs{s}^{\alpha_{m-1}-\alpha_{m}}\bigr)\bigr\},
\]
and where the second error term only occurs when $m>1$. This gives
\begin{equation}
\label{eq: example2}
	\Im \Psi(a+\I y) \sim \mp\sqrt{\frac{a_{m}}{b_{n}}}\sin\bigl((\alpha_{n}-\alpha_{m})\pi/4\bigr)\abs{y}^{-(\alpha_{n}-\alpha_{m})/2}, \quad \text{as $y\to\pm\infty$}.
\end{equation}
In the exponential-type distributed-order model, we have
\[
	 \Psi(s) = \sqrt{\tau}\biggl\{1 - \frac{\log\tau}{2\log s} + \frac{3(\log\tau)^{2}}{8(\log s)^{2}} + O\biggl(\frac{1}{\,\abs{\log s}^{3}}\biggr)\biggr\},
\]
yielding
\begin{equation}
\label{eq: example3}
	\Im \Psi(a+\I y) \sim \mp \frac{\sqrt{\tau}\pi\log(1/\tau)}{4} \frac{1}{(\log \abs{y})^{2}}, \quad \text{as $y\to\pm\infty$}.
\end{equation}

In these examples, $\Im \Psi(a+\I y)$ decays less quickly to zero than $\abs{y}^{-1}$, except in the first one in the case that $\alpha_{n}=1$ and $\alpha_{m}=0$. Inserting the obtained asymptotics in \eqref{eq: bound inverse Laplace K}, we obtain that the integrand of \eqref{eq: inverse Laplace K} decays like $\e^{-c\,\abs{x}y^{\epsilon}}$ for some positive $c$ and $\epsilon$, so that we may differentiate indefinitely with respect to $x$ and $t$ if $x\neq0$, showing that $K$ is smooth there.

This $\gtrsim\abs{y}^{-1+\epsilon}$-decay for $\Im \Psi(a+\I y)$ holds actually in general, save for some exceptional cases where one cannot in general expect smoothness of $K$ on the boundary of the light cone. For example, formula \eqref{eq: example1} in the case that $\alpha_{n}=1$ and $\alpha_{m}=0$ implies that for some $\delta>0$,
\[
	\tilde{S}(x,s) = \frac{\sqrt{\tau}}{2s}\biggl(1+\frac{\gamma}{2s} + O\bigl(\,\abs{s}^{-1-\delta}\bigr)\biggr)
	\exp\biggl(-\abs{x}\sqrt{\tau}s - \frac{\sqrt{\tau}\gamma\abs{x}}{2} + O\bigl(\,\abs{x}\abs{s}^{-\delta}\bigr)\biggr), \quad \gamma \coloneqq \frac{a_{m}}{a_{n}}-\frac{b_{m}}{b_{n}},
\]
from which one can show\footnote{See e.g. \cite[Theorem 5.1]{BO} where this is worked out in detail for the SLS case.} that 
\[
	S(x,t) = \frac{\sqrt{\tau}}{2}\exp\biggl(-\frac{\sqrt{\tau}\gamma}{2}\abs{x}\biggr)H(t-\sqrt{\tau}\abs{x}) + \text{continuous function},
\]
where $H$ denotes the Heaviside function. This shows that $S$ is not continuous on the boundary of the cone $\abs{x}\le t/\sqrt{\tau}$.
	
Let us first describe these exceptional cases in general.
\subsubsection*{Exceptional cases}Let $\tau>0$, $a>0$, $b\ge0$ and suppose $a/b\geq\tau$ (with again $a/0=\infty$). Let $\lambda$ be a positive Radon measure on $[0,1]$ with $\lambda(\{0\})=0$ and $\max\supp\lambda=1$. The exceptional cases are those pairs of measures $\mus, \mue$ of the form 
\begin{equation}
\label{eq: exceptional case}
	\mus(\alpha) = a\delta(\alpha) + \tau\lambda(\alpha), \quad \mue(\alpha) = b\delta(\alpha) + \lambda(\alpha).
\end{equation} 
Note that of the classical mechanical models \eqref{eq: classical models}, the Hooke model, the Maxwell model, and the Zener model are exceptional. It is known that for those models, the fundamental solution is discontinuous on the boundary of the cone $\abs{x} \le t/\sqrt{\tau}$ (see \cite[Theorem 5.1]{BO} for the Zener case; the same reasoning applies to the Maxwell case).
\bigskip

We now turn back to the general case. The following lemma will provide the estimate which will yield smoothness of $K$ for $x\neq0$. Then, in the Theorem \ref{Thm: Gevrey}, we will show Gevrey regularity in this set.   
\begin{lemma}
\label{lem: lower bound Im Psi}
Assume that $\mus,\mue$ are positive Radon measures satisfying \eqref{eq:TD} and $\Psi$ is defined as in \eqref{Psi}. 
Then 
\begin{equation}
\label{eq: sgn Im Psi}
	\sgn\Im\Psi(s) = -\sgn\Im s.
\end{equation}
If in addition $\mus,\mue$ are not exceptional, i.e. not of the form \eqref{eq: exceptional case}, then there exists an $\eta>-1$ so that uniformly for $\theta=\arg s \in [\pi/4, 3\pi/4]$:
\begin{equation}\label{eq: ImPsi}
	\abs{\Im\Psi(s)} \gtrsim \abs{s}^{\eta}, \quad \mbox{as } \abs{s}\to\infty.
\end{equation}
\end{lemma}
\begin{proof}
Write $s=R\e^{\I\theta}$ and $\Phi_{\sigma}(s)/\Phi_{\eps}(s) = u+\I v$. We have that 
\begin{align}
	v = \Im \frac{\Phi_{\sigma}(s)}{\Phi_{\eps}(s)} = -\frac{1}{\abs{\Phi_{\eps}(s)}^{2}}\iint_{\Delta}R^{\alpha+\beta}\sin\bigl((\alpha-\beta)\theta\bigr)\dif\nues(\alpha,\beta),	
\label{eq: v}
\end{align}
where again $\Delta = \{(\alpha,\beta)\in [0,1]: \alpha>\beta\}$.
By the thermodynamical restriction \eqref{eq:TD}, the first assertion immediately follows. 

The imaginary part of $\sqrt{u+\I v}$ can be written as follows:
\[
	\Im \sqrt{u+\I v} 	= \frac{\sgn v}{\sqrt{2}}\sqrt{\sqrt{u^{2}+v^{2}}-u} = \frac{\sgn v}{\sqrt{2}}\frac{\abs{v}}{\sqrt{\sqrt{u^{2}+v^{2}}+u}}.
\]
For the denominator we have by Proposition \ref{prop: existence tau} and by \eqref{eq: behavior Phi} and \eqref{eq: lower bound Phi} that for any $\delta>0$
\[
	\sqrt{\sqrt{u^{2}+v^{2}}+u} \lesssim \abs{\Psi(s)} \lesssim \min\{1, O_{\delta}(R^{(M_{\sigma}-M_{\eps})/2 + \delta})\}.
\]
We next derive a lower bound for $\abs{v}$, and consider three cases.

\bigskip

\noindent\textbf{Case 1.} $M_{\sigma} < M_{\eps}$.

\noindent Consider the set $V = [M_{\eps}-\delta,M_{\eps}] \times [\max\{M_{\sigma}-\delta,0\},M_{\sigma}]$, which is a subset of $\Delta$ if $\delta>0$ is small enough. If $\delta$ is so small that $M_{\sigma} < M_{\eps}-\delta$, then 
$$\nues(V) = \mus([\max\{M_{\sigma}-\delta,0\},M_{\sigma}])\mue([M_{\eps}-\delta,M_{\eps}]) > 0.$$ Hence 
\begin{align*}
	\abs{v} 	&\gtrsim R^{-2M_{\eps}}\iint_{V}R^{\alpha+\beta}\sin\bigl(\theta(\alpha-\beta)\bigr)\dif\nues(\alpha,\beta) \\
			&\ge \nues(V)\min\bigl\{\sin\bigl(M_{\eps}-\delta-M_{\sigma})\theta\bigr), \sin\theta\bigr\} R^{M_{\sigma}-M_{\eps}-2\delta}, 
\end{align*}
so that $\Im\Psi(s) \gtrsim_{\delta}R^{(M_{\sigma}-M_{\eps})/2-3\delta}$. Since $(M_{\sigma}-M_{\eps})/2 \ge -1/2$, the lemma follows in this case upon taking $\delta$ small enough.

\mbox{}

\noindent\textbf{Case 2.} $M_{\sigma} = M_{\eps} = M$ and $\forall x\in[0,M): F(x)>\tau$ (recall that $F$ is defined by \eqref{eq: def F}).

\noindent In this case we will even have $\abs{\Im \Psi(s)} \gtrsim_{\delta} R^{-\delta}$ for every $\delta>0$. Let $x<M$, and set $\epsilon = F(x)-\tau > 0$. Let $y\in(x,M)$ be such that $F(y) \le \tau+\epsilon/2$, and let $z\in(y,M)$ be such that $F(z)\le \tau+\epsilon/4$. Set $V= [z,M] \times [x,y)$; we will show that $\nues(V)>0$.
\begin{align*}
	\nues(V)	&= \mue([z,M])\mus([x,y)) - \mus([z, M])\mue([x,y)) \\
			&= \mue([z,M])\mue([x,M])F(x) - \mue([z,M])\mue([y,M])F(y) \\
			&\phantom{=} \quad - F(z)\mue([z,M])\mue([x, M]) + F(z)\mue([z,M])\mue([y,M]) \\
			&=\mue([z,M])\bigl(\mue([x,M])(F(x)-F(z)) - \mue([y,M])(F(y)-F(z))\bigr) \\
			&\ge \mue([z,M])\bigl(\mue([x, M]3\epsilon/4 - \mue([y,M])\epsilon/2\bigr) \\
			&\ge (\epsilon/4)\mue([z,M])\mue([x,M]) >0.
\end{align*}
Using this we get that 
\begin{align*}
	\abs{\Im \Psi(s)} 	&\gtrsim \abs{v} \gtrsim R^{-2M}\iint_{V}R^{\alpha+\beta}\sin\bigl((\alpha-\beta)\theta\bigr)\dif\nues(\alpha,\beta) \\
					&\ge \nues(V)\min\bigl\{\sin\bigl((z-y)\theta\bigr), \sin\theta\bigr\}R^{z + x-2M}.
\end{align*}
Here we may even take $x$ arbitrarily close to $M$, leading to the bound $\gtrsim_{\delta} R^{-\delta}$ for every $\delta>0$.

\mbox{}

\noindent\textbf{Case 3.} $M_{\sigma} = M_{\eps} = M$ and $\exists x\in [0,M)$ such that $F(x) = \tau$.

\noindent Note that by monotonicity, $F$ is necessarily constant and equal to $\tau$ on $[x,M)$. Let $y\in [0,M)$ be such that $F(y)>\tau$. (If $F(y)=\tau$ for every $y\ge0$, then $\mus=\tau\mue$ which is an exceptional case.) Then if $\delta < M-x$, we have $\mus([y,M-\delta)) > \tau\mue([y,M-\delta))$. Let then $V = [M-\delta/2, M] \times [y, M-\delta)$, and note that $\nues(V) = \mue([M-\delta/2, M])\mus([y,M-\delta)) - \tau\mue([M-\delta/2, M])\mue([y,M-\delta)) > 0$. Hence
\begin{align*}
	\abs{\Im \Psi(s)} 	&\gtrsim \abs{v} \gtrsim R^{-2M}\iint_{V}R^{\alpha+\beta}\sin\bigl((\alpha-\beta)\theta\bigr)\dif\nues(\alpha,\beta) \\
					&\ge \nues(V)\min\bigl\{\sin(\delta\theta/2),\sin\theta\bigr\}R^{-M-\delta/2+y}.
\end{align*}
If $M=1$, then we can take $y>0$ by assumption. Indeed, if $M=1$ and $F(y)=\tau$ for all $y>0$, we are in an exceptional case. Taking such strictly positive $y$ and $\delta$ such that $\delta/2 < y$ yields the result. If $M<1$, then we take $\delta$ such that $\delta/2 < 1-M$.
\end{proof}
Note that the worked out examples \eqref{eq: example1} (excluding the case $\alpha_{n}=1$, $\alpha_{m}=0$), \eqref{eq: example2}, and \eqref{eq: example3} are instances of \textbf{case 3}, \textbf{case 1}, and \textbf{case 2} respectively in the proof of the lemma. 


The estimate $\abs{\Im\Psi(s)}\gtrsim \abs{s}^{\eta}$ provided by the lemma will yield smoothness of $K$ for $x\neq0$. Actually, we will show Gevrey regularity in this set.  Recall that for  $\beta\ge 0$ and $\Omega \subseteq \R^{2}$ open, a function $\vphi\in C^{\infty}(\Omega)$ is in the Gevrey class of order $\beta$, i.e.  $\vphi\in G^{\beta}(\Omega)$,  if for every compact $A\subseteq\Omega$ there exists a constant $C=C_{A}>0$ such that 
\begin{equation}\label{Gc}
	\sup_{(x,t)\in A}\abs{\dpd[n]{}{x}\dpd[m]{}{t} \vphi(x,t)} \le C^{1+n+m}\Gamma(\beta(n+m+1)), \quad \text{for all $n, m \ge0$},
\end{equation}
where $\Gamma$ denotes the Gamma function.

\begin{theorem}\label{Thm: Gevrey} 
Assume that $\mus,\mue$ are positive Radon measures satisfying \eqref{eq:TD} which do not belong to the exceptional cases. Let $\Psi$ be defined as in \eqref{Psi} and  let $\eta>-1$ be such that \eqref{eq: ImPsi} holds.
Then on the set $\Omega = \{(x,t): x\neq0\}$, $K$ belongs to the Gevrey class $G^{\beta}(\Omega)$ of order $\beta=1/(1+\eta)$.
\end{theorem}
\begin{proof}
Set $\beta = 1/(1+\eta)$. We have to show that for any compact $A\subseteq \Omega$, there exists a positive constant $C=C_{A}$ such that \eqref{Gc} holds.
Fix some $a>0$ and let $C_{1}$ and $C_{2}$ be constants such that 
\[
	\abs{\Psi(a+\I y)} \le C_{1}, \quad\text{for $y\ge0$}, \qquad \Im \Psi(a+\I y) \le -C_{2}y^{\eta}, \quad \text{for $y\ge1$}.
\]
Since $K$ is even in $x$, we may suppose that $x>0$. Differentiating under the integral in \eqref{eq: inverse Laplace K}, we get
\begin{align*}
	\dpd[n]{}{x}\dpd[m]{}{t}&K(x,t) = \frac{(-1)^{n}}{4\pi\I}\int_{a-\I\infty}^{a+\I\infty}\Psi(s)^{n+1}s^{n+m}\exp\bigl(-xs\Psi(s) + ts\bigr)\dif s \\
					&\lesssim \e^{at}C_{1}^{n+1}2^{n+m}\biggl\{\int_{0}^{1}(a^{n+m}+y^{n+m})\dif y + \int_{1}^{\infty}(a^{n+m}+y^{n+m})\e^{-xC_{2}y^{1+\eta}}\dif y\biggr\} \\
					&\lesssim \e^{at}C_{1}^{n+1}2^{n+m}\biggl\{a^{n+m}+1 + \biggl(\frac{1}{xC_{2}}\biggr)^{\frac{n+m+1}{1+\eta}}\Gamma\biggl(\frac{n+m+1}{1+\eta}\biggr)\biggr\}.
\end{align*}
The result now easily follows.
\end{proof}
\begin{remark}
For exceptional models, the fundamental solution may or may not be smooth on $x\neq0$. For example, the lower bound $\Im \Psi(s) \gtrsim (\log \abs{s})^{2}/\abs{s}$ is already sufficient to obtain smoothness (although not of Gevrey type), while $\Im \Psi(s) \gtrsim \log \abs{s}/\abs{s}$ does not guarantee smoothness.
\end{remark}

On the half line $x=0, t\ge0$ the fundamental solution $K$ is not smooth:
\begin{proposition}
For proper fractional models, the half-line $\{0\}\times[0,\infty)$ is a set of discontinuity for $\partial_{x}K$.
\end{proposition}
\begin{proof}
Using representation \eqref{eq: K in cone} we get 
\[
	\dpd{K}{x}(0^{\pm}, t) = \mp \frac{1}{2\pi}\int_{0}^{\infty}\Im \frac{\Phi_{\sigma}(q\e^{\I\pi})}{\Phi_{\eps}(q\e^{\I\pi})} q\e^{-qt}\dif q.
\]
In order to show that $\partial_{x} K$ is discontinuous at $(0,t)$, it suffices to show that the imaginary part in the above integral is non-zero on some set of positive measure. We have
\[
	\Im \frac{\Phi_{\sigma}(q\e^{\I\pi})}{\Phi_{\eps}(q\e^{\I\pi})} = -\frac{1}{\abs{\Phi_{\eps}(q\e^{\I\pi})}^{2}}\iint_{\Delta}q^{\alpha+\beta}\sin\bigl((\alpha-\beta)\pi\bigr)\dif\nues(\alpha,\beta).
\]
To show that this is non-zero, we proceed as in Lemma \ref{lem: lower bound Im Psi}. Extra care has to be taken, since the sine in the integrand may become zero, in contrast to the situation in Lemma \ref{lem: lower bound Im Psi} where $\theta\in[\pi/4,3\pi/4]$.

\mbox{}

\noindent\textbf{Case 1.} $M_{\sigma} < M_{\eps}$. Considering again the set $V = [M_{\eps}-\delta, M_{\eps}] \times [\max\{M_{\sigma}-\delta, 0\}, M_{\sigma}]$ for small enough $\delta>0$ yields
\[
	\iint_{\Delta}q^{\alpha+\beta}\sin\bigl((\alpha-\beta)\pi\bigr)\dif\nues(\alpha,\beta) \ge C\mus([\max\{M_{\sigma}-\delta, 0\}, M_{\sigma}])\mue([M_{\eps}-\delta, M_{\eps}])q^{M_{\eps}-\delta},
\]	
with 
\[
	C= \min\bigl\{\sin\bigl((M_{\eps}-M_{\sigma}-\delta)\pi\bigr), \sin\bigl((M_{\eps}-\max\{M_{\sigma}-\delta, 0\})\pi)\bigr\}.
\]
The constant $C$ is strictly positive unless $M_{\sigma}=0$ and $M_{\eps}=1$. In this case we can take $V=[\delta, 1-\delta] \times \{0\}$ and get
\[
	\iint_{\Delta}q^{\alpha+\beta}\sin\bigl((\alpha-\beta)\pi\bigr)\dif\nues(\alpha,\beta) \ge \mus(\{0\})\mue([\delta,1-\delta])\sin(\delta\pi)q^{\delta}.
\]
Since we assume to be working with proper fractional models, $\mue([\delta,1-\delta])>0$ for $\delta$ sufficiently small (otherwise this would be the Voigt model).

\bigskip

\noindent\textbf{Case 2.} $M_{\sigma} = M_{\eps} = M$ and $\forall x\in[0,M): F(x)>\tau$.
In this case we may again take $V=[z,M]\times[x,y)$ with $x<y<z<M$ and all sufficiently close to $M$ such that $\nues(V) > 0$. We get
\[
	\iint_{\Delta}q^{\alpha+\beta}\sin\bigl((\alpha-\beta)\pi\bigr)\dif\nues(\alpha,\beta) \ge \nues(V)\sin\bigl((z-y)\pi\bigr)q^{z+x}.
\]

\bigskip

\noindent\textbf{Case 3.} $M_{\sigma} = M_{\eps} = M$ and $\exists x\in [0,M)$ such that $F(x) = \tau$. We consider again $V = [M-\delta/2,M]\times[y,M-\delta)$ for small enough $\delta>0$ and $y$ satisfying $F(y)>\tau$. We get
\[
	\iint_{\Delta}q^{\alpha+\beta}\sin\bigl((\alpha-\beta)\pi\bigr)\dif\nues(\alpha,\beta) \ge \nues(V)\min\bigl\{\sin\bigl((M-y)\pi\bigr),\sin(\delta\pi/2)\bigr\}q^{y+M-\delta/2}.
\]
This works unless $M=1$ and the only $y$ satisfying $F(y)>\tau$ is $y=0$. However, in this case we have $\mue([\delta,1-\delta])>0$ for sufficiently small $\delta$ (otherwise we would be in the Zener case). Furthermore $F(0)>\tau$ together with $F(x)=\tau$, $x>0$ implies $\mus(\{0\})>0$, so we can take $V=[\delta,1-\delta]\times\{0\}$ to get the lower bound $\mus(\{0\})\mue([\delta, 1-\delta])\sin(\delta\pi)q^{\delta}$.
\end{proof}

We can conclude the following. In all cases, the fundamental solution $K$ is real analytic on the set $0<\abs{x}<ct$, $c=1/k$ (and also on $\abs{x}>ct$, $K$ being zero there). If we are not in an exceptional case, then $K$ is also smooth on the boundary of the cone: it actually belongs to the Gevrey class $G^{\beta}$ for some $\beta>0$ on the set $x\neq0$. In the exceptional case, $K$ may or may not be smooth on the boundary. On the other hand, for proper fractional models, $K$ is not of class $C^{1}$ on the half-line $x=0$, $t\ge0$.

\section{Wave velocities}\label{Sec: QA}
Given a pair of positive Radon measures $(\mus,\mue)$ on $[0,1]$, consider the pair of measures $(\tilde{\mu}_{\sigma},\tilde{\mu}_{\eps})$ defined as
\[
	\tilde{\mu}_{\sigma}(A) = \mue(1-A)=\mue(\{1-\alpha:\alpha\in A\}), \quad \tilde{\mu}_{\eps}(A)=\mus(1-A),
\]
where $A$ is a Borel subset of $[0,1]$. They are also a pair of positive Radon measures on $[0,1]$, and by considering rectangles one immediately sees that $(\mus,\mue)$ satisfies \eqref{eq:TD} if and only if $(\tilde{\mu}_{\sigma},\tilde{\mu}_{\eps})$ satisfies \eqref{eq:TD}. For the corresponding ``$\Phi$-functions'' we have
\[
	\tilde{\Phi}_{\sigma}(s) = \int_{0}^{1}s^{\alpha}\dif\tilde{\mu}_{\sigma}(\alpha) = \int_{0}^{1}s^{1-\alpha}\dif\mue(\alpha) = s\Phi_{\eps}(1/s), \quad \tilde{\Phi}_{\eps}(s) = s\Phi_{\sigma}(1/s).
\]
If $(\mus, \mue)$ is a pair of measures satisfying \eqref{eq:TD}, Proposition \ref{prop: existence tau} guarantees the existence of some $\tau \in [0,\infty)$ such that $\Phi_{\sigma}(s)/\Phi_{\eps}(s) \to \tau$ as $s\to \infty$. Applying the same Proposition to the measures $(\tilde{\mu}_{\sigma}, \tilde{\mu}_{\eps})$ yields the existence of a number $\rho \in (0,\infty]$ such that $\Phi_{\sigma}(s)/\Phi_{\eps}(s) \to \rho$ as $s\to 0$. Here, $\rho=1/\tilde{\tau}$ with $\tilde{\tau} = \lim_{s\to\infty}\tilde{\Phi}_{\sigma}(s)/\tilde{\Phi}_{\eps}(s)$.

In any case we have the inequality $\rho \ge \tau$. To see this, let $m_{\sigma}=\min\supp \mus= 1-\max \supp\tilde{\mu}_{\eps}$ and $m_{\eps}=\min\supp\mue=1-\max\supp\tilde{\mu}_{\sigma}$. We have $m_{\sigma}\le m_{\eps}$. If $m_{\sigma}<m_{\eps}$, then $\tilde{\tau}=0$ and $\rho=\infty$ and the inequality holds. Similarly, if $M_{\sigma}<M_{\eps}$ then $\tau=0$ and the inequality again holds. We may hence assume that $M_{\sigma}=M_{\eps}=M$ and $m_{\sigma}=m_{\eps}=m$. From the discussion in Section \ref{Sec: existence and uniqueness}, we see that  
\[
	\tau = \lim_{x\to M} F(x) = \lim_{x\to M}\frac{\mus([x,M])}{\mue([x,M])}, \quad \rho = \lim_{x\to m} G(x) \coloneqq \lim_{x\to m}\frac{\mus([m,x])}{\mue([m,x])},
\]
with $G(x) = 1/\tilde{F}(1-x)$. If $M=m$, then $\mus(\alpha) = a\delta(\alpha-M)$, $\mue(\alpha)=b\delta(\alpha-M)$ for $a,b>0$ and $\tau=\rho=a/b$. If $m<M$, then we apply \eqref{eq:TD} to see that for $\delta<(M-m)/2$, $G(m+\delta) \ge F(M-\delta)$. Indeed, this follows by letting $V=[M-\delta, M]\times[m, m+\delta]\subseteq \Delta$ and noting that $\nues(V)\ge0$. Hence $\rho\ge \tau$.

The constants $\tau$ and $\rho$ can be expressed in terms of material constants of the viscoelastic body as in \cite{ZO2020}, see also\cite{Mainardi}. Denote by $J(t)$ the strain response to a unit step of stress, i.e.\ $\eps(t)=J(t)$ when $\sigma(t)=H(t)$. By \eqref{eq: Laplace transform const. eq. measures}, we have that $\tilde{J}(s) = (\Phi_{\sigma}(s)/\Phi_{\eps}(s))\cdot (1/s)$. The function $J(t)$ is referred to as the \emph{creep compliance}, its limiting value at $0$, $J_{g} \coloneqq \lim_{t\to0^{+}}J(t)$, is called the \emph{glass compliance}, and its limiting value at $\infty$, $J_{e}\coloneqq \lim_{t\to\infty}J(t)$, is called the \emph{equilibrium compliance}. In view of the identities $\lim_{t\to 0^{+}}f(t) = \lim_{s\to\infty}s\tilde{f}(s)$ and $\lim_{t\to \infty}f(t) = \lim_{s\to0}s\tilde{f}(s)$, we get
\[
	J_{g} = \tau, \quad J_{e} = \rho.
\]
Similarly one defines the \emph{relaxation modulus} $G(t)$ as the stress response to a unit step of strain, and its limiting values the \emph{glass modulus} $G_{g}$ and the \emph{equilibrium modulus} $G_{e}$. We have $\tilde{G}(s) = \Phi_{\eps}(s)/\Phi_{\sigma}(s)\cdot (1/s)$, and 
\[
	G_{g} = \frac{1}{J_{g}} = \frac{1}{\tau}, \quad G_{e} = \frac{1}{J_{e}} = \frac{1}{\rho}.
\]

\bigskip

Finally we relate the constants $\tau$ and $\rho$ to wave speeds. For this we introduce a notion of weak velocity.
\begin{definition}
Let $F(x,t)$ be a function defined for real $x$ and positve $t$, and let $\mathfrak{F}_{t}(\lambda) = tF(\lambda t, t)$. We say that $F$ has \emph{weak equilibrium velocity} $v_{e}$ if there is some constant $c\neq 0$ such that $\mathfrak{F}_{t}(\lambda) \to c\delta(\lambda-v_{e})$ in the space of distributions $\D'$, as $t\to\infty$. This means that for every smooth function $\vphi$ with compact support, $\int_{-\infty}^{\infty}\mathfrak{F}_{t}(\lambda)\vphi(\lambda)\dif\lambda \to c\vphi(v_{e})$.

\noindent Similarly $F$ has \emph{weak initial velocity} $v_{i}>0$ if $\mathfrak{F}_{t}(\lambda) \to c\delta(\lambda-v_{i})$ in $\mathcal{D}'$ as $t\to 0^{+}$.
\end{definition}
For a wave packet to have weak equilibrium velocity $v_{e}$, it means that for sufficiently large time, most of the mass of the wave packet is concentrated around the point $x=v_{e}t$. However, this ``concentration'' might be rather weak, since the packet is allowed to spread out on spacial scales up to $o(t)$. For example, the dispersive Gaussian wave packet 
\[
	F(x,t) = \frac{1}{\sqrt{2\pi t}}\exp\biggl(-\frac{(x-vt)^{2}}{2t}\biggr)
\]
with mean $vt$ and standard deviation $\sqrt{t}$ has weak equilibrium velocity $v_{e}=v$.

We now compute these weak velocities for the function $K(x,t) = \partial_{t}S(x,t)$. This is the solution to the Cauchy problem \eqref{eq: DFWE}-\eqref{eq: Cauchy data} with initial data $u_{0}(x) = \delta(x)$, $v_{0}(x)=0$. Being an even function, it suffices to study $K$ for positive $x$. Recall that we denote the Heaviside function by $H$.
\begin{theorem}
Let $\mus$ and $\mue$ be measures satisfying \eqref{eq:TD} which form a proper fractional model, and which do not belong to the exceptional cases. Let $\tau$ and $\rho$ be the corresponding limiting values of $\Phi_{\sigma}/\Phi_{\eps}$. If $\tau>0$, then $H(x)K(x,t)$ has weak initial velocity $v_{i}=1/\sqrt{\tau}$. If $\rho < \infty$, then $H(x)K(x,t)$ has weak equilibrium velocity $v_{e}=1/\sqrt{\rho}$.
\end{theorem}
\begin{proof}
By the assumptions the conclusions of Theorem \ref{th: support S} and Lemma \ref{lem: lower bound Im Psi} hold. Hence the integral \eqref{eq: inverse Laplace K} converges absolutely for $x\neq0$, and we may shift the contour of integration to the imaginary axis. Indeed, in view of \eqref{eq: bound near 0}, $\I y\Psi(\I y)$ is a continuous function of $y$, and the possible singularity of $\Psi(\I y)$ at $y=0$ is integrable. Hence for $x>0$:
\[
	K(x,t) = \frac{1}{4\pi}\int_{-\infty}^{\infty}\Psi(\I y)\exp\bigl(-\I yx\Psi(\I y)+\I ty\bigr)\dif y.
\]
Set $\mathfrak{K}_{t}(\lambda)=tH(\lambda)K(\lambda t, t)$, and let $\vphi\in\mathcal{D}$. Suppose first that $\supp \vphi \subseteq (0,\infty)$. We want to switch the order of integration in $\int_{0}^{\infty} \mathfrak{K}_{t}(\lambda)\vphi(\lambda)\dif \lambda$, which will be justified by the Fubini-Tonelli theorem. In view of \eqref{eq: bound near 0} and Lemma \ref{lem: lower bound Im Psi}, let $C_{1}, C_{2}>0$ be such that 
\[	
	\abs{\Psi(\I y)} \le \frac{C_{1}}{\sqrt{y}}, \quad 0<y<1, \qquad \Im \Psi(\I y) \le -C_{2}y^{\eta}, \quad y\ge1,
\]
where $\eta\in(-1,0)$ (recall that $\Im \Psi(\I y)\to0$ as $y\to\infty$, so $\eta<0$). Then 
\begin{multline*}
\int_{0}^{\infty}\int_{-\infty}^{\infty}\abs{\vphi(\lambda)\Psi(\I y)\exp\bigl(-\I yt\lambda\Psi(\I y)+\I ty\bigr)}\dif y\dif \lambda \\
	\lesssim \int_{0}^{\infty}\abs{\vphi(\lambda)}\biggl(\int_{0}^{1}\frac{C_{1}}{\sqrt{y}}\dif y + \int_{1}^{\infty}\e^{-C_{2}t \lambda y^{1+\eta}}\dif y\biggr)\dif\lambda
	\lesssim \int_{0}^{\infty}\abs{\vphi(\lambda)}\bigl(1+\lambda^{-\frac{1}{1+\eta}}\bigr)\dif \lambda < \infty.
\end{multline*}
Hence we may interchange the order of integration. We get
\begin{align}
	\int_{0}^{\infty}\mathfrak{K}_{t}(\lambda)\vphi(\lambda)\dif\lambda	&= \frac{t}{4\pi}\int_{-\infty}^{\infty}\Psi(\I y)\e^{\I ty}\int_{0}^{\infty}\vphi(\lambda)\e^{-\I y t\Psi(\I y)\lambda}\dif\lambda \nonumber\\
		&=\frac{t}{4\pi}\int_{-\infty}^{\infty}\Psi(\I y)\tilde{\vphi}\bigl(\I y t\Psi(\I y)\bigr)\e^{\I ty}\dif y \nonumber\\
		&= \frac{1}{4\pi}\int_{-\infty}^{\infty}\Psi(\I y/t)\tilde{\vphi}\bigl(\I y\Psi(\I y/t)\bigr)\e^{\I y}\dif y. \label{eq: int K phi}
\end{align}

Assume now that $\tau>0$. We will take the limit for $t\to 0$ in the above integral by applying the dominated convergence theorem. By \eqref{eq: sgn Im Psi}, $\Re\bigl(\I y\Psi(\I y/t)\bigr)\ge0$. Hence, for $\abs{y}\le 1$, the integrand is dominated by 
\[
	(1+\sqrt{t/\abs{y}})\sup_{\Re s\ge0}\abs{\tilde{\vphi}(s)} \le (1+\sqrt{t/\abs{y}})\int_{0}^{\infty}\abs{\vphi(\lambda)}\dif \lambda.
\]
For $\abs{y}\ge1$ and $t$ sufficiently small we have that $\Psi(\I y/t)$ is bounded and that $\Re \Psi(\I y/t) \in [\sqrt{\tau}/2, 3\sqrt{\tau}/2]$, so the integrand in \eqref{eq: int K phi} is dominated by
\[
	\sup_{\substack{\Re s \ge0 \\ \abs{\Im s} \in [\frac{\sqrt{\tau}}{2}y,\frac{3\sqrt{\tau}}{2}y]}}\abs{\tilde{\vphi}(s)} \lesssim \frac{1}{y^{2}},
\]
which follows from integrating by parts twice in $\tilde{\vphi}(s) = \int_{0}^{\infty}\vphi(\lambda)\e^{-s\lambda}\dif\lambda$. We conclude that 
\[
	\lim_{t\to 0^{+}} \int_{0}^{\infty}\mathfrak{K}_{t}(\lambda)\vphi(\lambda)\dif\lambda = \frac{1}{4\pi}\int_{-\infty}^{\infty}\sqrt{\tau}\tilde{\vphi}(\I\sqrt{\tau}y)\e^{\I y}\dif y 
	= \frac{1}{2}\vphi\biggl(\frac{1}{\sqrt{\tau}}\biggr).
\]

It remains to treat the case when $0\in \supp\vphi$. For this we will use the representation \eqref{eq: K in cone}. Let now $C_{1}$ and $C_{2}$ be constants such that $\abs{\Psi(q\e^{\pm\I\pi})} \le C_{1}/\sqrt{q}+C_{2}$, and suppose that\footnote{The general case can be reduced to these two cases: for $\vphi\in\D$, one can write $\vphi=\vphi_{1}+\vphi_{2}$ with $0\not\in\supp\vphi_{1}$ and $\supp\vphi_{2}\subseteq [-1/(2C_{2}), 1/(2C_{2})]$.} $\supp\vphi\subseteq[-1/(2C_{2}),1/(2C_{2})]$. Note that $C_{2}\ge \sqrt{\tau}$. We show that $\mathfrak{K}_{t}(\lambda)$ converges boundedly to $0$ on $[0, 1/(2C_{2})]$, so that $\int_{0}^{\infty}\mathfrak{K}_{t}(\lambda)\vphi(\lambda)\dif\lambda \to 0$. For $\lambda\le 1/(2C_{2}) <1/\sqrt{\tau}$ we have
\begin{equation}
\label{eq: K(lambda) in cone}
	\mathfrak{K}_{t}(\lambda) = \frac{1}{4\pi\I}\int_{0}^{\infty}\bigl(\Psi(q\e^{-\I\pi}/t)\e^{\lambda q\Psi(q\e^{-\I\pi}/t)} -  \Psi(q\e^{\I\pi}/t)\e^{\lambda q\Psi(q\e^{\I\pi}/t)}\bigr)\e^{-q}\dif q.
\end{equation}
The integrand is dominated by $(C_{1}/\sqrt{q}+C_{2})\e^{-q/2+C_{1}/(2C_{2})\sqrt{q}}$. By dominated convergence, $\mathfrak{K}_{t}(\lambda)$ converges pointwise to $0=\int_{0}^{\infty}(\sqrt{\tau}\e^{\lambda\sqrt{\tau}q} - \sqrt{\tau}\e^{\lambda\sqrt{\tau}q})\e^{-q}\dif q$ and is bounded by 
$$2\int_{0}^{\infty}(C_{1}/\sqrt{q}+C_{2})\e^{-q/2+C_{1}/(2C_{2})\sqrt{q}}\dif q$$ on the interval $[0,1/(2C_{2})]$. Hence by bounded convergence,
\[
	\int_{0}^{1/(2C_{2})}\mathfrak{K}_{t}(\lambda)\vphi(\lambda)\dif\lambda \to 0 \quad \text{as $t\to0$}.
\]

The proof for showing that $K$ has weak equilibrium velocity $1/\sqrt{\rho}$ when $\rho<\infty$ is analogous. Consider now constants $C_{1}$ and $C_{2}$ such that $\abs{\Psi(s)}\le C_{1}$ and $\Im \Psi(\I y) \le -C_{2}y^{\eta}$ for $y\ge1$ with $\eta \in (-1,0)$, and note that $C_{1}\ge \rho$.

Suppose first that $\vphi\in \D$ with $\supp\vphi \subseteq [\lambda_{0},\lambda_{1}]$ for some $0<\lambda_{0}<\lambda_{1}$. We want to apply dominated convergence to take the limit for $t\to\infty$ in \eqref{eq: int K phi}. For $\abs{y}\le 1$, the integrand is clearly bounded. By Lemma \ref{lem: condition A4}, $\Psi$ is non-zero on the line segment $[-\I,\I]$, so there is some $\epsilon>0$ such that $\abs{\Psi(\I u)}\ge\epsilon$ for $\abs{u}\le1$. For $u\ge0$, clearly
\[
	\arg \Phi_{\sigma}(\I u) \in [0,\pi/2], \quad \arg\Phi_{\eps}(\I u) \in [0,\pi/2],
\]
so $\arg\Psi(\I u) \in [-\pi/4,\pi/4]$. This implies that $\Re\Psi(\I u) \ge \epsilon/\sqrt{2}$ for $0\le u\le1$. Hence for $1\le \abs{y}\le t$, the integrand of \eqref{eq: int K phi} is dominated by
\[
	C_{1} \sup_{\substack{\Re s \ge0 \\ \abs{\Im s} \ge \epsilon y/\sqrt{2}}}\abs{\tilde{\vphi}(s)} \lesssim \frac{1}{y^{2}},
\]
like before. For $\abs{y}\ge t$, we have $-y\Im \Psi(\I y/t) \ge C_{2}\abs{y}^{1+\eta}t^{-\eta}$. Hence in that range we can dominate the integrand by
\[
	C_{1} \sup_{\Re s\ge C_{2}\,\abs{y}^{1+\eta}t^{-\eta}}\abs{\tilde{\vphi}(s)} \lesssim \e^{-C_{2}\lambda_{0}\,\abs{y}^{1+\eta}}.
\]
In any case, the integrand is dominated by $1/y^{2} + \e^{-C_{2}\lambda_{0}\,\abs{y}^{1+\eta}}$ in the range $\abs{y}\ge1$, and we may again apply dominated convergence to see that 
\[
	\int_{0}^{\infty}\mathfrak{K}_{t}(\lambda)\vphi(\lambda)\dif \lambda \to \frac{1}{4\pi}\int_{-\infty}^{\infty}\sqrt{\rho}\tilde{\vphi}(\I\sqrt{\rho}y)\e^{\I y}\dif y 
	= \frac{1}{2}\vphi\biggl(\frac{1}{\sqrt{\rho}}\biggr), \quad \text{as $t\to\infty$}.
\]

If $0\le \lambda \le 1/(2C_{1})$, we see from \eqref{eq: K(lambda) in cone} that 
\[
	\mathfrak{K}_{t}(\lambda) \lesssim C_{1}\int_{0}^{\infty}\e^{-q/2}\dif q, \quad \mathfrak{K}_{t}(\lambda) \to 0, \quad \text{as $t\to \infty$}.
\]
If now $\vphi\in\D$ with $\supp \vphi \subseteq [-1/(2C_{1}), 1/(2C_{1})]$, then $\int_{0}^{\infty}\mathfrak{K}_{t}(\lambda)\vphi(\lambda)\dif\lambda\to0$ as $t\to\infty$, by bounded convergence. This concludes the proof.
\end{proof}
Note that since $\rho\ge \tau$, the weak equilibrium velocity $1/\sqrt{\rho}$ of $K$ is less than or equal to the ``wave front velocity'' $1/\sqrt{\tau}$ (recall that from Theorem \ref{th: support S}, $K$ is supported in the cone $\abs{x}\le t/\sqrt{\tau}$). The weak initial velocity equals this wave front velocity.

For a more detailed asymptotic analysis of the function $K$ for some particular models (namely the fractional Zener model), we refer to \cite[Section 4]{BO}.

\section{Conclusion}\label{Sec: Concl}
In this paper we model wave propagation in viscoelastic media by means of the equations \eqref{eq: NL}, \eqref{eq: const. eq. measures}, and \eqref{eq: sm} with $\mus$ and $\mue$ positive Radon measures supported on $[0,1]$. We impose the thermodynamical condition \eqref{eq:TD} on the involved measures, from which we derive existence and uniqueness of solutions for the corresponding distributed-order fractional wave equation \eqref{eq: DFWE1}.

The fundamental solution is supported in the cone $\abs{x} \le ct$ and is real analytic for $x\neq0$ and $\abs{x}\neq ct$. For a large subclass of models, it is Gevrey regular on the set $x\neq 0$. For proper fractional models, the fundamental solution is not $C^{1}$ on the half-line $x=0$, $t\ge0$. 

We define the weak initial velocity $v_{i}$ and the weak equilibrium velocity $v_{e}$, measuring the ``wave packet speed'' at small and large times, respectively. They are evaluated as
\[
	v_{i} = \frac{1}{\sqrt{\tau}} = \lim_{s\to\infty}\frac{1}{\Psi(s)}=c, \quad v_{e} = \frac{1}{\sqrt{\rho}} = \lim_{s\to0}\frac{1}{\Psi(s)},
\]
with $\Psi$ given by \eqref{Psi}.
One can also relate these velocities to the material constants of the viscoelastic body.

\end{document}